\documentclass[11pt]{amsart}
\usepackage[dvipsnames]{xcolor}
\usepackage[utf8]{inputenc}
\usepackage{amssymb,amsmath,amsthm,amsfonts,fullpage,float,latexsym,bbm,microtype,cite,fancyvrb}
\usepackage{tikz}
\usepackage{enumerate}
\usetikzlibrary{decorations.pathreplacing}
\usepackage{hyperref}
\hypersetup{colorlinks=true, linkcolor=blue, citecolor=magenta, filecolor=magenta, urlcolor=magenta}

\usepackage[colorinlistoftodos]{todonotes}
\definecolor{andresblue}{rgb}{0,0.72,0.92}
\definecolor{andrespink}{rgb}{1,0,1}

\definecolor{munsell}{rgb}{0.0, 0.5, 0.69}

\makeatletter
\newtheorem*{rep@theorem}{\rep@title}\newcommand{\newreptheorem}[2]{%
\newenvironment{rep#1}[1]{%
\def\rep@title{\bf #2 \ref{##1}}%
\begin{rep@theorem}}%
{\end{rep@theorem}}}
\makeatother
\newreptheorem{theorem}{Theorem}

\newtheorem{theorem}{Theorem}[section]
\newtheorem{proposition}[theorem]{Proposition}

\newtheorem{conjecture}[theorem]{Conjecture}
\newtheorem{lemma}[theorem]{Lemma}
\newtheorem{corollary}[theorem]{Corollary}
\theoremstyle{definition}
\newtheorem{remark}[theorem]{Remark}
\newtheorem{definition}[theorem]{Definition}

\newtheorem{example}[theorem]{Example}
\newtheorem{problem}[theorem]{Problem}

\newcommand{\Q}{\mathbb{Q}}
\newcommand{\R}{\mathbb{R}}
\newcommand{\Z}{\mathbb{Z}}

\newcommand{\bx}{\mathbf{x}}
\newcommand{\X}{\mathfrak{X}}
\newcommand{\by}{\mathbf{y}}
\newcommand{\Xn}{\mathfrak{X}_{n}(a,b)}
\newcommand{\PF}{\mathsf{PF}}
\DeclareMathOperator{\Vol}{Vol}
\DeclareMathOperator{\nVol}{NVol}
\DeclareMathOperator{\Pyr}{Pyr}
\DeclareMathOperator{\Ver}{Ver}

\begin{document}

\title{Generalized parking function polytopes}

\author{Mitsuki Hanada}
\thanks{Hanada is partially supported by the Funai Overseas Scholarship.}
\address{Department of Mathematics\\
         University of California, Berkeley\\
\url{https://math.berkeley.edu/~mhanada/}}
\email{mhanada@berkeley.edu}

\author{John Lentfer}
\thanks{Lentfer is supported by the National Science Foundation Graduate Research Fellowship DGE-2146752.}
\address{Department of Mathematics\\
         University of California, Berkeley\\
\url{https://math.berkeley.edu/~jlentfer/}}
\email{jlentfer@berkeley.edu}

\author{Andr\'es R. Vindas-Mel\'endez}
\thanks{Vindas-Mel\'endez is partially supported by the National Science Foundation under Award DMS-2102921.}
\address{Department of Mathematics\\
         University of California, Berkeley\\
\url{https://math.berkeley.edu/~vindas}}
\email{andres.vindas@berkeley.edu}


\begin{abstract}
A classical parking function of length $n$ is a list of positive integers $(a_1, a_2, \ldots, a_n)$ whose nondecreasing rearrangement $b_1 \leq b_2 \leq \cdots \leq b_n$ satisfies $b_i \leq i$.
The convex hull of all parking functions of length $n$ is an $n$-dimensional polytope in $\mathbb{R}^n$, which we refer to as the classical parking function polytope. 
Its geometric properties have been explored in (Amanbayeva and Wang 2022) in response to a question posed in (Stanley 2020).
We generalize this family of polytopes by studying the geometric properties of the convex hull of $\mathbf{x}$-parking functions for $\mathbf{x}=(a,b,\dots,b)$, which we refer to as $\mathbf{x}$-parking function polytopes. 
We explore connections between these $\mathbf{x}$-parking function polytopes, the Pitman-Stanley polytope, and  the partial permutahedra of (Heuer and Striker 2022).
In particular, we establish a closed-form expression for the volume of $\mathbf{x}$-parking function polytopes.
This allows us to answer a conjecture of (Behrend et al. 2022) and also obtain a new closed-form expression for the volume of the convex hull of classical parking functions as a corollary.
\end{abstract}

\maketitle

\section{Introduction}

A \textit{classical parking function} of length $n$ is a list $(a_1, a_2, \ldots, a_n)$ of positive integers whose nondecreasing rearrangement $b_1 \leq b_2 \leq \cdots \leq b_n$ satisfies $b_i \leq i$. 
It is well-known that the number of classical parking functions of length $n$ is $(n+1)^{n-1}$; this number surfaces in a variety of places, for example, it counts the number of planted forests on $n$ vertices 
and the number of regions of a Shi arrangement (see \cite{Bon} for further discussion).
Let $\PF_n$ denote the convex hull of all parking functions of length $n$ in $\R^n$.
In 2020, Stanley \cite{Sta} asked for the number of vertices, the number of faces, the number of lattice points $ \PF_n \cap \Z^n$, and the volume of $\PF_n$. 
These questions were first answered by Amanbayeva and Wang \cite{AW} and Stong \cite{Sto}.

In Section \ref{sec:classical}, we revisit the classical parking function polytope $\PF_n$. 
We provide new results on the appearance of both lower-dimensional parking function polytopes and permutahedra as facets of $\PF_n$.
We connect the classical parking function polytope with the recent work of Heuer and Striker \cite{HS} and Behrend et al. \cite{BCC} on partial permutahedra, and show when they are integrally equivalent. 
We collect different characterizations of the normalized volume of $\PF_n$ in Theorem \ref{thm:main_theorem} and give a new simple, closed-form answer to Stanley's original question on the volume of $\PF_n$.

A natural direction is to extend these results for generalizations of parking functions.
We consider $\mathbf{x}$-parking functions, where $\mathbf{x}=(x_1,\ldots,x_n)$ is a vector of positive integers, which have been explored from an enumerative perspective previously by Yan \cite{Yan2, Yan, Bon} and Pitman and Stanley \cite{PitmanStanley}.
In Section \ref{sec:x-parking}, we focus on the case when $\mathbf{x} = (a,b,b,\ldots, b)$ and generalize results of Amanbayeva and Wang \cite{AW}. 
We establish in Theorem \ref{thm:generalized_closed_form_volume} a closed-form normalized volume formula for all positive integers $a,b$.

\begin{theorem}\label{thm:generalized_closed_form_volume}
For any positive integers $a,b,n$, the normalized volume of the $\bx$-parking function polytope $\nVol(\mathfrak{X}_{n}(a,b))$ is given by
\begin{equation}
    \nVol(\mathfrak{X}_{n}(a,b)) = -n!\left(\frac{b}{2}\right)^n \sum_{i=0}^n \binom{n}{i} (2i-3)!! \left(2n-1 + \frac{2a-2}{b}\right)^{n-i}. \nonumber
\end{equation}
\end{theorem}
\noindent The proof of Theorem \ref{thm:generalized_closed_form_volume} uses tools from analytic combinatorics and analytic number theory, e.g., Ramanujan's Master Theorem.  
By determining when partial permutohedra are integrally equivalent to $\mathbf{x}$-parking functions, we also prove a conjecture of Behrend et al. \cite{BCC}, as Corollary \ref{cor:conjecture}.

In Section \ref{sec:weakly_increasing}, we introduce weakly increasing $\mathbf{x}$-parking functions, which are $\mathbf{x}$-parking functions that are in weakly increasing order $a_1 \leq a_2 \leq \cdots \leq a_n$ without any rearrangement. 
We explore a subpolytope of the $\mathbf{x}$-parking function polytope which is constructed as the convex hull of weakly increasing $\mathbf{x}$-parking functions. 
We show that the convex hull of weakly increasing $\mathbf{x}$-parking functions is integrally equivalent to certain Pitman-Stanley polytopes. 
In the literature, the Pitman-Stanley polytope has been called the ``parking function polytope'' as its volume is related to parking functions.
However, in this paper, we call the convex hull of any generalization of parking functions a parking function polytope.

We conclude with Section \ref{sec:future} where we present some routes for further study on unimodular triangulations, Ehrhart theory, other generalizations where $\mathbf{x} \neq (a, b, \ldots, b)$, and rational parking functions.


\section{The classical parking function polytope}\label{sec:classical}
The classical parking function polytope was first studied in \cite{AW, Sto}. As mentioned in the introduction, a \textit{classical parking function} of length $n$ is a list $(a_1, a_2, \ldots, a_n)$ of positive integers whose nondecreasing rearrangement $b_1 \leq b_2 \leq \cdots \leq b_n$ satisfies $b_i \leq i$. 
Let $\PF_n$ denote the convex hull of all parking functions of length $n$ in $\R^n$, which we call the \textit{classical parking function polytope}.
For example, the classical parking functions of length $3$ are $111$, $112$, $121$, $211$, $113$, $131$, $311$, $122$, $212$, $221$, $123$, $132$, $231$, $213$, $312$, and $321$.
For $\PF_3$, their convex full in $\R^3$, see Figure \ref{fig:PF_3} (left).

\begin{figure}[h]
    \centering
    \begin{tikzpicture}%
	[x={(-0.702073cm, -0.395494cm)},
	y={(0.711985cm, -0.374594cm)},
	z={(-0.013069cm, 0.838608cm)},
	scale=1.9,
	back/.style={loosely dotted, thin},
	edge/.style={color=black, thick},
	facet/.style={fill=andresblue,fill opacity=0.500000},
	vertex/.style={inner sep=1pt,circle,draw=andrespink,fill=andrespink,thick}]
%
%

\coordinate (1.00000, 1.00000, 1.00000) at (1.00000, 1.00000, 1.00000);
\coordinate (3.00000, 2.00000, 1.00000) at (3.00000, 2.00000, 1.00000);
\coordinate (1.00000, 1.00000, 3.00000) at (1.00000, 1.00000, 3.00000);
\coordinate (3.00000, 1.00000, 2.00000) at (3.00000, 1.00000, 2.00000);
\coordinate (3.00000, 1.00000, 1.00000) at (3.00000, 1.00000, 1.00000);
\coordinate (1.00000, 2.00000, 3.00000) at (1.00000, 2.00000, 3.00000);
\coordinate (1.00000, 3.00000, 1.00000) at (1.00000, 3.00000, 1.00000);
\coordinate (1.00000, 3.00000, 2.00000) at (1.00000, 3.00000, 2.00000);
\coordinate (2.00000, 3.00000, 1.00000) at (2.00000, 3.00000, 1.00000);
\coordinate (2.00000, 1.00000, 3.00000) at (2.00000, 1.00000, 3.00000);
\draw[edge,back] (1.00000, 1.00000, 1.00000) -- (1.00000, 1.00000, 3.00000);
\draw[edge,back] (1.00000, 1.00000, 1.00000) -- (3.00000, 1.00000, 1.00000);
\draw[edge,back] (1.00000, 1.00000, 1.00000) -- (1.00000, 3.00000, 1.00000);
\node[vertex] at (1.00000, 1.00000, 1.00000)     {};
\fill[facet] (1.00000, 3.00000, 2.00000) -- (1.00000, 2.00000, 3.00000) -- (2.00000, 1.00000, 3.00000) -- (3.00000, 1.00000, 2.00000) -- (3.00000, 2.00000, 1.00000) -- (2.00000, 3.00000, 1.00000) -- cycle {};
\fill[facet] (2.00000, 3.00000, 1.00000) -- (1.00000, 3.00000, 1.00000) -- (1.00000, 3.00000, 2.00000) -- cycle {};
\fill[facet] (2.00000, 1.00000, 3.00000) -- (1.00000, 1.00000, 3.00000) -- (1.00000, 2.00000, 3.00000) -- cycle {};
\fill[facet] (3.00000, 1.00000, 1.00000) -- (3.00000, 2.00000, 1.00000) -- (3.00000, 1.00000, 2.00000) -- cycle {};
\draw[edge] (3.00000, 2.00000, 1.00000) -- (3.00000, 1.00000, 2.00000);
\draw[edge] (3.00000, 2.00000, 1.00000) -- (3.00000, 1.00000, 1.00000);
\draw[edge] (3.00000, 2.00000, 1.00000) -- (2.00000, 3.00000, 1.00000);
\draw[edge] (1.00000, 1.00000, 3.00000) -- (1.00000, 2.00000, 3.00000);
\draw[edge] (1.00000, 1.00000, 3.00000) -- (2.00000, 1.00000, 3.00000);
\draw[edge] (3.00000, 1.00000, 2.00000) -- (3.00000, 1.00000, 1.00000);
\draw[edge] (3.00000, 1.00000, 2.00000) -- (2.00000, 1.00000, 3.00000);
\draw[edge] (1.00000, 2.00000, 3.00000) -- (1.00000, 3.00000, 2.00000);
\draw[edge] (1.00000, 2.00000, 3.00000) -- (2.00000, 1.00000, 3.00000);
\draw[edge] (1.00000, 3.00000, 1.00000) -- (1.00000, 3.00000, 2.00000);
\draw[edge] (1.00000, 3.00000, 1.00000) -- (2.00000, 3.00000, 1.00000);
\draw[edge] (1.00000, 3.00000, 2.00000) -- (2.00000, 3.00000, 1.00000);
\node[vertex] at (3.00000, 2.00000, 1.00000)     {};
\node[vertex] at (1.00000, 1.00000, 3.00000)     {};
\node[vertex] at (3.00000, 1.00000, 2.00000)     {};
\node[vertex] at (3.00000, 1.00000, 1.00000)     {};
\node[vertex] at (1.00000, 2.00000, 3.00000)     {};
\node[vertex] at (1.00000, 3.00000, 1.00000)     {};
\node[vertex] at (1.00000, 3.00000, 2.00000)     {};
\node[vertex] at (2.00000, 3.00000, 1.00000)     {};
\node[vertex] at (2.00000, 1.00000, 3.00000)     {};
\end{tikzpicture}
\qquad\qquad
%
\begin{tikzpicture}[x  = {(0.9cm,-0.076cm)},
                    y  = {(-0.06cm,0.95cm)},
                    z  = {(-0.44cm,-0.29cm)},
                    scale = 1.5,
                    ]

  \definecolor{pointcolor_p}{rgb}{ 1,0,1 }
  \tikzstyle{pointstyle_p} = [fill=pointcolor_p]

  \coordinate (v0_p) at (1, 1, 4);
  \coordinate (v1_p) at (1, 4, 1);
  \coordinate (v2_p) at (4, 1, 1);
  \coordinate (v3_p) at (0.25, 0.25, 0.25);
  \coordinate (v4_p) at (1, 3, 4);
  \coordinate (v5_p) at (1, 4, 3);
  \coordinate (v6_p) at (3, 1, 4);
  \coordinate (v7_p) at (3, 4, 1);
  \coordinate (v8_p) at (4, 1, 3);
  \coordinate (v9_p) at (4, 3, 1);
  \coordinate (v10_p) at (0.333333, 0.333333, 1.33333);
  \coordinate (v11_p) at (0.333333, 1.33333, 0.333333);
  \coordinate (v12_p) at (1.33333, 0.333333, 0.333333);
  \coordinate (v13_p) at (0.25, 0.25, 0.75);
  \coordinate (v14_p) at (0.25, 0.75, 0.25);
  \coordinate (v15_p) at (0.75, 0.25, 0.25);
  \coordinate (v16_p) at (2, 3, 4);
  \coordinate (v17_p) at (2, 4, 3);
  \coordinate (v18_p) at (3, 2, 4);
  \coordinate (v19_p) at (3, 4, 2);
  \coordinate (v20_p) at (4, 2, 3);
  \coordinate (v21_p) at (4, 3, 2);
  \coordinate (v22_p) at (0.5, 1.5, 2);
  \coordinate (v23_p) at (0.5, 2, 1.5);
  \coordinate (v24_p) at (1.5, 0.5, 2);
  \coordinate (v25_p) at (1.5, 2, 0.5);
  \coordinate (v26_p) at (2, 0.5, 1.5);
  \coordinate (v27_p) at (2, 1.5, 0.5);
  \coordinate (v28_p) at (0.333333, 0.666667, 1.33333);
  \coordinate (v29_p) at (0.333333, 1.33333, 0.666667);
  \coordinate (v30_p) at (0.666667, 0.333333, 1.33333);
  \coordinate (v31_p) at (0.666667, 1.33333, 0.333333);
  \coordinate (v32_p) at (1.33333, 0.333333, 0.666667);
  \coordinate (v33_p) at (1.33333, 0.666667, 0.333333);
  \coordinate (v34_p) at (0.25, 0.5, 0.75);
  \coordinate (v35_p) at (0.25, 0.75, 0.5);
  \coordinate (v36_p) at (0.5, 0.25, 0.75);
  \coordinate (v37_p) at (0.5, 0.75, 0.25);
  \coordinate (v38_p) at (0.75, 0.25, 0.5);
  \coordinate (v39_p) at (0.75, 0.5, 0.25);

  \definecolor{edgecolor_p}{rgb}{ 0,0,0 }


  \tikzstyle{facestyle_p} = [fill=andresblue,fill opacity=0.500000, draw=edgecolor_p, line width=1 pt, line cap=round, line join=round]

  \draw[facestyle_p] (v6_p) -- (v0_p) -- (v10_p) -- (v30_p) -- (v24_p) -- (v6_p) -- cycle;
  \draw[facestyle_p] (v2_p) -- (v8_p) -- (v26_p) -- (v32_p) -- (v12_p) -- (v2_p) -- cycle;
  \draw[facestyle_p] (v8_p) -- (v6_p) -- (v24_p) -- (v26_p) -- (v8_p) -- cycle;
  \draw[facestyle_p] (v0_p) -- (v4_p) -- (v22_p) -- (v28_p) -- (v10_p) -- (v0_p) -- cycle;
  \draw[facestyle_p] (v36_p) -- (v13_p) -- (v3_p) -- (v15_p) -- (v38_p) -- (v36_p) -- cycle;
  \draw[facestyle_p] (v30_p) -- (v10_p) -- (v13_p) -- (v36_p) -- (v30_p) -- cycle;
  \draw[facestyle_p] (v24_p) -- (v30_p) -- (v36_p) -- (v38_p) -- (v32_p) -- (v26_p) -- (v24_p) -- cycle;

  \fill[pointcolor_p] (v24_p) circle (1 pt);
  \node at (v24_p) [text=black, inner sep=0.5pt, above right, draw=none, align=left] {};
  \fill[pointcolor_p] (v30_p) circle (1 pt);
  \node at (v30_p) [text=black, inner sep=0.5pt, above right, draw=none, align=left] {};
  \fill[pointcolor_p] (v36_p) circle (1 pt);
  \node at (v36_p) [text=black, inner sep=0.5pt, above right, draw=none, align=left] {};
  \fill[pointcolor_p] (v26_p) circle (1 pt);
  \node at (v26_p) [text=black, inner sep=0.5pt, above right, draw=none, align=left] {};

  \draw[facestyle_p] (v5_p) -- (v1_p) -- (v11_p) -- (v29_p) -- (v23_p) -- (v5_p) -- cycle;
  \draw[facestyle_p] (v12_p) -- (v32_p) -- (v38_p) -- (v15_p) -- (v12_p) -- cycle;

  \fill[pointcolor_p] (v32_p) circle (1 pt);
  \node at (v32_p) [text=black, inner sep=0.5pt, above right, draw=none, align=left] {};
  \fill[pointcolor_p] (v38_p) circle (1 pt);
  \node at (v38_p) [text=black, inner sep=0.5pt, above right, draw=none, align=left] {};

  \draw[facestyle_p] (v35_p) -- (v14_p) -- (v3_p) -- (v13_p) -- (v34_p) -- (v35_p) -- cycle;
  \draw[facestyle_p] (v4_p) -- (v5_p) -- (v23_p) -- (v22_p) -- (v4_p) -- cycle;
  \draw[facestyle_p] (v10_p) -- (v28_p) -- (v34_p) -- (v13_p) -- (v10_p) -- cycle;

  \fill[pointcolor_p] (v10_p) circle (1 pt);
  \node at (v10_p) [text=black, inner sep=0.5pt, above right, draw=none, align=left] {};
  \fill[pointcolor_p] (v13_p) circle (1 pt);
  \node at (v13_p) [text=black, inner sep=0.5pt, above right, draw=none, align=left] {};

  \draw[facestyle_p] (v29_p) -- (v11_p) -- (v14_p) -- (v35_p) -- (v29_p) -- cycle;
  \draw[facestyle_p] (v23_p) -- (v29_p) -- (v35_p) -- (v34_p) -- (v28_p) -- (v22_p) -- (v23_p) -- cycle;

  \fill[pointcolor_p] (v23_p) circle (1 pt);
  \node at (v23_p) [text=black, inner sep=0.5pt, above right, draw=none, align=left] {};
  \fill[pointcolor_p] (v29_p) circle (1 pt);
  \node at (v29_p) [text=black, inner sep=0.5pt, above right, draw=none, align=left] {};
  \fill[pointcolor_p] (v35_p) circle (1 pt);
  \node at (v35_p) [text=black, inner sep=0.5pt, above right, draw=none, align=left] {};
  \fill[pointcolor_p] (v34_p) circle (1 pt);
  \node at (v34_p) [text=black, inner sep=0.5pt, above right, draw=none, align=left] {};
  \fill[pointcolor_p] (v28_p) circle (1 pt);
  \node at (v28_p) [text=black, inner sep=0.5pt, above right, draw=none, align=left] {};
  \fill[pointcolor_p] (v22_p) circle (1 pt);
  \node at (v22_p) [text=black, inner sep=0.5pt, above right, draw=none, align=left] {};

  \draw[facestyle_p] (v9_p) -- (v2_p) -- (v12_p) -- (v33_p) -- (v27_p) -- (v9_p) -- cycle;
  \draw[facestyle_p] (v1_p) -- (v7_p) -- (v25_p) -- (v31_p) -- (v11_p) -- (v1_p) -- cycle;
  \draw[facestyle_p] (v37_p) -- (v39_p) -- (v15_p) -- (v3_p) -- (v14_p) -- (v37_p) -- cycle;

  \fill[pointcolor_p] (v3_p) circle (1 pt);
  \node at (v3_p) [text=black, inner sep=0.5pt, above right, draw=none, align=left] {};

  \draw[facestyle_p] (v7_p) -- (v9_p) -- (v27_p) -- (v25_p) -- (v7_p) -- cycle;
  \draw[facestyle_p] (v33_p) -- (v12_p) -- (v15_p) -- (v39_p) -- (v33_p) -- cycle;

  \fill[pointcolor_p] (v12_p) circle (1 pt);
  \node at (v12_p) [text=black, inner sep=0.5pt, above right, draw=none, align=left] {};
  \fill[pointcolor_p] (v15_p) circle (1 pt);
  \node at (v15_p) [text=black, inner sep=0.5pt, above right, draw=none, align=left] {};

  \draw[facestyle_p] (v11_p) -- (v31_p) -- (v37_p) -- (v14_p) -- (v11_p) -- cycle;

  \fill[pointcolor_p] (v11_p) circle (1 pt);
  \node at (v11_p) [text=black, inner sep=0.5pt, above right, draw=none, align=left] {};
  \fill[pointcolor_p] (v14_p) circle (1 pt);
  \node at (v14_p) [text=black, inner sep=0.5pt, above right, draw=none, align=left] {};

  \draw[facestyle_p] (v25_p) -- (v27_p) -- (v33_p) -- (v39_p) -- (v37_p) -- (v31_p) -- (v25_p) -- cycle;

  \fill[pointcolor_p] (v25_p) circle (1 pt);
  \node at (v25_p) [text=black, inner sep=0.5pt, above right, draw=none, align=left] {};
  \fill[pointcolor_p] (v27_p) circle (1 pt);
  \node at (v27_p) [text=black, inner sep=0.5pt, above right, draw=none, align=left] {};
  \fill[pointcolor_p] (v33_p) circle (1 pt);
  \node at (v33_p) [text=black, inner sep=0.5pt, above right, draw=none, align=left] {};
  \fill[pointcolor_p] (v39_p) circle (1 pt);
  \node at (v39_p) [text=black, inner sep=0.5pt, above right, draw=none, align=left] {};
  \fill[pointcolor_p] (v37_p) circle (1 pt);
  \node at (v37_p) [text=black, inner sep=0.5pt, above right, draw=none, align=left] {};
  \fill[pointcolor_p] (v31_p) circle (1 pt);
  \node at (v31_p) [text=black, inner sep=0.5pt, above right, draw=none, align=left] {};

  \draw[facestyle_p] (v17_p) -- (v19_p) -- (v7_p) -- (v1_p) -- (v5_p) -- (v17_p) -- cycle;

  \fill[pointcolor_p] (v1_p) circle (1 pt);
  \node at (v1_p) [text=black, inner sep=0.5pt, above right, draw=none, align=left] {};

  \draw[facestyle_p] (v9_p) -- (v21_p) -- (v20_p) -- (v8_p) -- (v2_p) -- (v9_p) -- cycle;

  \fill[pointcolor_p] (v2_p) circle (1 pt);
  \node at (v2_p) [text=black, inner sep=0.5pt, above right, draw=none, align=left] {};

  \draw[facestyle_p] (v19_p) -- (v21_p) -- (v9_p) -- (v7_p) -- (v19_p) -- cycle;

  \fill[pointcolor_p] (v9_p) circle (1 pt);
  \node at (v9_p) [text=black, inner sep=0.5pt, above right, draw=none, align=left] {};
  \fill[pointcolor_p] (v7_p) circle (1 pt);
  \node at (v7_p) [text=black, inner sep=0.5pt, above right, draw=none, align=left] {};

  \draw[facestyle_p] (v18_p) -- (v16_p) -- (v4_p) -- (v0_p) -- (v6_p) -- (v18_p) -- cycle;

  \fill[pointcolor_p] (v0_p) circle (1 pt);
  \node at (v0_p) [text=black, inner sep=0.5pt, above right, draw=none, align=left] {};

  \draw[facestyle_p] (v16_p) -- (v17_p) -- (v5_p) -- (v4_p) -- (v16_p) -- cycle;

  \fill[pointcolor_p] (v5_p) circle (1 pt);
  \node at (v5_p) [text=black, inner sep=0.5pt, above right, draw=none, align=left] {};
  \fill[pointcolor_p] (v4_p) circle (1 pt);
  \node at (v4_p) [text=black, inner sep=0.5pt, above right, draw=none, align=left] {};

  \draw[facestyle_p] (v20_p) -- (v18_p) -- (v6_p) -- (v8_p) -- (v20_p) -- cycle;

  \fill[pointcolor_p] (v6_p) circle (1 pt);
  \node at (v6_p) [text=black, inner sep=0.5pt, above right, draw=none, align=left] {};
  \fill[pointcolor_p] (v8_p) circle (1 pt);
  \node at (v8_p) [text=black, inner sep=0.5pt, above right, draw=none, align=left] {};

  \draw[facestyle_p] (v20_p) -- (v21_p) -- (v19_p) -- (v17_p) -- (v16_p) -- (v18_p) -- (v20_p) -- cycle;

  \fill[pointcolor_p] (v20_p) circle (1 pt);
  \node at (v20_p) [text=black, inner sep=0.5pt, above right, draw=none, align=left] {};
  \fill[pointcolor_p] (v21_p) circle (1 pt);
  \node at (v21_p) [text=black, inner sep=0.5pt, above right, draw=none, align=left] {};
  \fill[pointcolor_p] (v19_p) circle (1 pt);
  \node at (v19_p) [text=black, inner sep=0.5pt, above right, draw=none, align=left] {};
  \fill[pointcolor_p] (v17_p) circle (1 pt);
  \node at (v17_p) [text=black, inner sep=0.5pt, above right, draw=none, align=left] {};
  \fill[pointcolor_p] (v16_p) circle (1 pt);
  \node at (v16_p) [text=black, inner sep=0.5pt, above right, draw=none, align=left] {};
  \fill[pointcolor_p] (v18_p) circle (1 pt);
  \node at (v18_p) [text=black, inner sep=0.5pt, above right, draw=none, align=left] {};


\end{tikzpicture}

    \caption{On the left, we have the classical parking function polytope $\PF_3$, where the hexagonal facet is the regular permutahedron $\Pi_3$ and the three triangular facets are copies of $\PF_2$. 
    On the right, we have the the Schlegel diagram of $\PF_4$. }
    \label{fig:PF_3}
\end{figure}
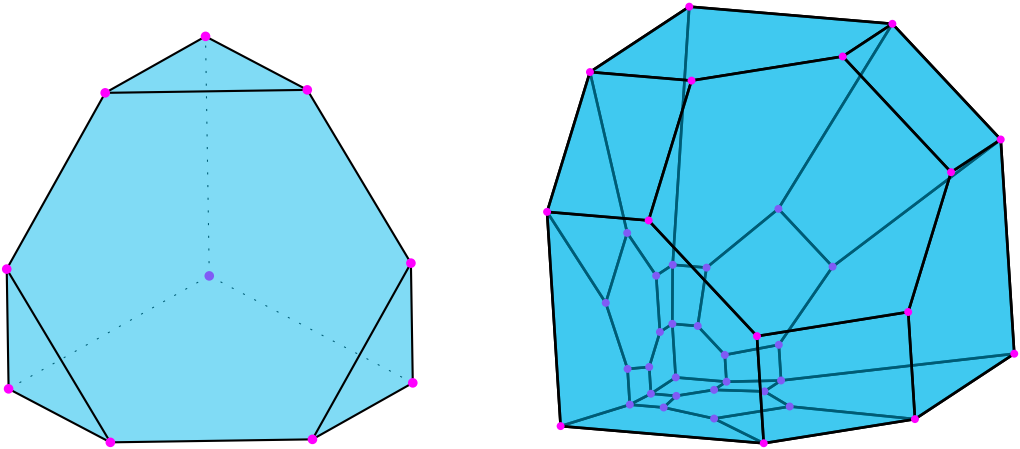

A convex polytope can be described by its vertex and hyperplane descriptions, which we briefly state here. Both were known to Stanley (see \cite{AW}). The vertex description of $\PF_n$ is the convex hull of all vertices given by permutations of \[ (\underbrace{1, \ldots, 1}_k, \underbrace{k+1, k+2, \ldots, n-2, n-1}_{n-k}),\] for $1 \leq k \leq n$. For a proof, this follows from the more general Proposition \ref{prop:abn_vertex} (taking $a=b=1$).

The hyperplane description of $\PF_n$: 
\begin{align*}
    1 \leq x_i &\leq n, &\text{ for } 1 \leq i \leq n,\\
    x_i+x_j &\leq n + (n-1), &\text{ for } i < j,\\
    x_i+x_j+x_k &\leq n + (n-1) + (n-2), &\text{ for } i < j < k,\\
    &\vdots\nonumber\\
    x_{i_1}+ x_{i_2}+ \cdots + x_{i_{n-2}} &\leq n + (n-1) + \cdots + 3, &\text{ for } i_1 < i_2 < \cdots < i_{n-2},\\
    x_{1}+ x_{2}+ \cdots + x_{n} &\leq n + (n-1) + \cdots + 1.
\end{align*}
For a proof, this follows from the more general Proposition \ref{prop:inequality} (taking $a=b=1$).

\subsection{Face structure}
We are able to say more about its face structure by using the regular permutahedron, which we now define.

\begin{definition}[Example 0.10, \cite{Ziegler}; Definition 2.1, \cite{Pos}]\label{def:permutahedron}
The \textit{regular permutahedron} $\Pi_{n} \subseteq \R^n$ is the $(n-1)$-dimensional polytope obtained as the convex hull of all vectors obtained by permuting the coordinates of the vector $(1,2,\ldots,n)$. 
Its vertices can be identified with the permutations in $S_n$ by associating with $(x_1, x_2, \ldots, x_n)$ the permutation that maps $x_i \mapsto i$ such that two permutations are adjacent if and only if the corresponding permutations differ by an adjacent transposition.
More generally, for $\mathbf{r}:=(r_1, \ldots, r_n) \in \R^n$, a \textit{permutahedron} $\Pi_n(\mathbf{r})$ is the convex hull of all vectors that are obtained by permuting the coordinates of the vector $\mathbf{r}$.

The defining inequalities for $\Pi_n$ (Proposition 2.5, \cite{Pos}; \cite{Rad}) are of the form \[x_1 +\cdots +x_n = \frac{n(n+1)}{2},\] and for all nonempty subsets $\{i_1, \ldots, i_k\} \subseteq \{1, \ldots, n\}$, \[x_{i_1} + \cdots +x_{i_k} \leq n + \cdots + (n-k+1)
\] 
\end{definition}

\begin{example}
Consider the classical parking function polytope $\PF_3$, and observe that one facet is the hexagon with vertices $(1,2,3)$, $(1,3,2)$, $(2,3,1)$, $(2,1,3)$, $(3,1,2)$, and $(3,2,1)$.
This aligns exactly with the description of the regular permutahedron $\Pi_3$.
Note that $\Pi_3$ is the intersection of $\PF_3$ with the supporting hyperplane given by the linear functional $x_1+x_2+x_3 = \frac{3(3+1)}{2}=6$. See Figure \ref{fig:PF_3} (left).
\end{example}

\begin{example}
For $\PF_4$, there is only one three-dimensional face that is the convex hull of 24 vertices, which are exactly the $4!$ permutations of $(1,2,3,4)$; it is $\Pi_4$. 
Similarly, we can find that there are exactly 8 two-dimensional faces which are the convex hulls of 6 vertices, which are, of course, the 8 hexagonal faces (and thus 2-dimensional permutahedra) which are facets of the three-dimensional permutahedra. See Figure \ref{fig:PF_3} (right) for the Schlegel diagram of $\PF_4$.
\end{example}

We have the following result on when the permutahedron appears as a facet of the classical parking function polytope.

\begin{proposition}\label{prop:permutahedron}
The regular permutahedron appears as a facet of the classical parking function polytope exactly once. 
\end{proposition}

\begin{proof}
By the definition of the parking function polytope $\PF_n$, it is the convex hull of all vertices given by permutations of \[ (\underbrace{1, \ldots, 1}_k, \underbrace{k+1, k+2, \ldots, n-2, n-1}_{n-k}),\] for $1 \leq k \leq n$ (here, $k=0$ would be superfluous).
Each permutation of $(1,2,\ldots,n)$ appears as a vertex. 
Thus, the convex hull of these $n!$ vertices, which is exactly $\Pi_n$, is contained within $\PF_n$.

We now use the hyperplane description of $\PF_n$.
Consider the hyperplane $H$ defined by $x_{1}+ x_{2}+ \cdots + x_{n} = \frac{n(n+1)}{2} $.
We claim  $H\cap \PF_n = \Pi_{n}$.
Since all permutations of $(1,2,\ldots, n)$ satisfy $1+2+\cdots+n = \frac{n(n+1)}{2}$, the vertices that give the vertex description of $\Pi_{n}$ are in $H \cap \PF_n$.
By taking their convex hull, it follows that $\Pi_n\subseteq H \cap \PF_n$.    

Now, suppose $\bx:=(x_1, x_2, \ldots, x_n)$ is a point in $H \cap \PF_n$.
As we are dealing with a subset of a polytope intersecting a hyperplane, $H \cap \PF_n$ is a polytope of dimension (at most) $n-1$. 
Suppose towards a contradiction that $\bx$ is not in $\Pi_{n}$. 
This means that by the defining inequalities for $\Pi_n$, 
\begin{enumerate}
    \item  $x_1 + \cdots + x_n \neq \frac{n(n+1)}{2}$, a contradiction, or
    \item there exists some nonempty subset $\{i_1, \ldots, i_k\} \subseteq \{1, \ldots, n\}$ such that 
    \[x_{i_1} + \cdots + x_{i_k} > n + \cdots + (n-k+1) 
    \]
\end{enumerate}
But (2) is not allowed by the defining inequalities for a parking function polytope. 
Hence, $\mathbf{x}$ is in $\Pi_n$, so $\Pi_n = H \cap \PF_n$ is a facet of $\PF_n$.

To show uniqueness, it suffices to show that the only hyperplane in the inequality description that corresponds to a facet with $n!$ vertices is $H$.
Assume a facet with $n!$ vertices corresponds to a hyperplane of the form \[x_{i_1}+\cdots + x_{i_k} = n + \cdots + (n+1-k), \text{ where } 1\leq k\leq n-2. \]
By the vertex description of $\PF_n$, the number of vertices that satisfy this equation is given by 
$$k! \sum_{m=1}^{n-k}\frac{(n-k)!}{m!},$$ 
where $m$ is the number of coordinates in the vertex that have value 1. 
Note that since $1\leq k \leq n-2$, it follows that $ n = \binom{n}{1} \leq \binom{n}{k}$.
Rearranging the inequality, we get that $n\cdot k!(n-k)!\leq n!$.
We can also see that $$\frac{1}{1}+\frac{1}{2!}+\cdots +\frac{1}{(n-k)!} < (n-k) < n.$$ 
Hence, $$k! (n-k)! \sum\limits_{m=1}^{n-k}\frac{1}{m!}< n\cdot k!(n-k)! \leq n!.$$
Therefore, $H$ is the only possible hyperplane in the inequality description of $\PF_n$ that corresponds to a facet with $n!$ many variables.
\end{proof}

The following lemma allows us to establish a lower bound for the number of permutahedra of any dimension in the parking function polytope.

\begin{lemma}
The $(n-2)$-dimensional permutahedron $\Pi_{n-1}$ appears as a facet of $\Pi_{n}$ exactly $2n$ times.
\end{lemma}

\begin{proof}
Consider the defining inequalities of $\Pi_{n}$ given in Definition \ref{def:permutahedron}.
We claim that the hyperplanes that correspond to $\Pi_{n-1}$ as facets are those of the form \[x_i \leq n \text{ and } x_{i_1} + \cdots + x_{i_{n-1}}\leq n + (n-1) + \cdots +2.\] 
We can see that for any $i$, the vertices of $\Pi_n$ that satisfy $x_i = n$ are those where the $i$-th coordinate is $n$ and all the other coordinates can be written as a permutation of $(1,2,\dots, n-1)$. 
This facet is exactly $\Pi_{n-1}$.
Similarly, we can see the vertices that satisfy  \[x_{i_1} + \cdots + x_{i_{n-1}} = n + (n-1) + \cdots +2, \text{ where } \{i_1,\dots, i_{n-1}\} = \{1,\ldots, n\} \setminus \{k\}, \text{ for  some } k,\] are vertices where the $k$-th coordinate is 1 and all the other coordinates can be written as a permutation of $(2,\dots, n)$. 
Thus, facets that correspond to hyperplanes of these forms are $\Pi_{n-1}$.

To show that these are the only hyperplanes that can correspond to $\Pi_{n-1}$, we will show that the only hyperplanes that can correspond to a facet with $(n-1)!$ vertices are the ones mentioned above.  
The proof is similar to the uniqueness proof of Proposition \ref{prop:permutahedron}.

For a hyperplane $x_{i_1} + \cdots +x_{i_k} = n + \cdots + (n-k+1)$ consisting of $k$ many variables, we can see that there are $k!(n-k)!$ vertices of $\Pi_n$ that satisfy it. If we have that $k$ is neither equal to $1$ nor $n-1$, it follows that $k!(n-k)!<(n-1)!$. 
Hence, if a hyperplane with $k$ many variables corresponds to $\Pi_{n-1}$, which has $(n-1)!$ vertices, it follows that $k$ equals $1$ or $n-1$.
\end{proof}

\begin{proposition}
The $(n-1)$-dimensional parking function polytope $\PF_{n-1}$ appears as a facet of the $n$-dimensional parking function polytope $\PF_n$ exactly $n$ times.
\end{proposition}

\begin{proof}
Consider the inequality description of $\PF_n$:
\begin{align}
    1 \leq x_i &\leq n, &\text{ for } 1 \leq i \leq n,\label{eq:1}\\
    x_i+x_j &\leq n + (n-1), &\text{ for } i < j,\label{eq:2}\\
    x_i+x_j+x_k &\leq n + (n-1) + (n-2), &\text{ for } i < j < k,\label{eq:3}\\
    &\vdots\nonumber\\
    x_{i_1}+ x_{i_2}+ \cdots + x_{i_{n-2}} &\leq n + (n-1) + \cdots + 3, &\text{ for } i_1 < i_2 < \cdots < i_{n-2},\label{eq:n-2}\\
    x_{1}+ x_{2}+ \cdots + x_{n} &\leq n + (n-1) + \cdots + 1.\label{eq:n}
\end{align}
Fix $x_i = n$ for some $i$; without loss of generality say $i=n$. 
Then we proceed to reduce the system of inequalities.
For all $1 \leq i \leq n-1$, we still have $1 \leq x_i$, but if any \[x_i > n - 1, \text{ then } x_i + x_n >  n + (n-1),\] which contradicts (\ref{eq:2}).
Thus, we have $1 \leq x_i \leq n-1$ for $1 \leq i \leq n-1$. 
Next, for $i < j < n$, if \[x_i + x_j > (n-1) + (n-2), \text{ then } x_i + x_j + x_n > n + (n-1) + (n-2),\] which contradicts (\ref{eq:3}). 
Thus, we have $x_i + x_j \leq (n-1) + (n-2)$ for $i < j < n$. 
Continuing this process, we can refine the inequalities up through: for $i_1 < i_2 < \cdots < i_{n-3} < n$, if \[x_{i_1} + x_{i_2} + \cdots + x_{i_{n-3}} > (n-1) + (n-2) + \cdots + 3,\] then \[x_{i_1} + x_{i_2} + \cdots + x_{i_{n-3}} + x_n > n + (n-1) + (n-2) + \cdots +3,\] which contradicts (\ref{eq:n-2}). 
Thus, we have \[x_{i_1} + x_{i_2} + \cdots + x_{i_{n-3}} \leq (n-1) + (n-2) + \cdots + 3.\] 
Lastly, consider if $x_1 + x_2 + \cdots + x_{n-1} > (n-1) + \cdots +1$. 
Then \[x_1 + x_2 + \cdots + x_{n-1} + x_n >  n+ (n-1) + \cdots +1,\] which contradicts (\ref{eq:n}). 
Thus, \[x_1 + x_2 + \cdots + x_{n-1} \leq (n-1) + \cdots +1.\]
Collecting these results gives the following inequality description.
\begin{align*}
    1 \leq x_i &\leq n-1, &\text{ for } 1 \leq i \leq n-1,\\
    x_i+x_j &\leq (n-1) + (n-2), &\text{ for } i < j < n,\\
    x_i+x_j+x_k &\leq (n-1) + (n-2) + (n-3), &\text{ for } i < j < k < n,\\
    &\vdots\\
    x_{i_1}+ x_{i_2}+ \cdots + x_{i_{n-3}} &\leq (n-1) + (n-2) + \cdots + 3, &\text{ for } i_1 < i_2 < \cdots < i_{n-2} < n,\\
    x_{1}+ x_{2}+ \cdots + x_{n-1} &\leq  (n-1) + \cdots + 1.
\end{align*}
This is exactly the inequality description for $\PF_{n-1}$. 
Hence, $\PF_{n-1}$ is a facet of $\PF_n$.
Now, as the choice of $i$ was arbitrary, there are $n$ choices for $i$, so there are $n$ copies of $\PF_{n-1}$ appearing as facets of $\PF_n$. 

Now we will show that these are the only occurrences. 
We know that $\PF_{n-1}$ has $(n-1)!\sum\limits_{m=1}^{n-1} \frac{1}{m!}$ many  vertices.
Similar to the proof of Proposition \ref{prop:permutahedron}, we can see that for a hyperplane with $k$ variables, there are $k!(n-k)!\sum\limits_{m=1}^{n-k} \frac{1}{m!}$ many vertices that satisfy it. 
If $k$ is not equal to $1$ or $n$, we have $n\leq \binom{n}{k}$, thus $k!(n-k)!\leq (n-1)!$.
Then since $$k!(n-k)!\sum\limits_{m=1}^{n-k} \frac{1}{m!} < k!(n-k)!\sum\limits_{m=1}^{n-1} \frac{1}{m!} \leq (n-1)! \sum\limits_{m=1}^{n-1} \frac{1}{m!},$$  these hyperplanes cannot correspond to a facet with $(n-1)!\sum\limits_{m=1}^{n-1} \frac{1}{m!}$ many vertices.
If $k =n$, then the vertices that satisfy the equation are exactly all permutations of $(1,2,\dots, n)$, which is $n! > (n-1)! \sum\limits_{m=1}^{n-1} \frac{1}{m!}$ many vertices, hence this hyperplane cannot correspond to $\PF_{n-1}$. 

For a hyperplane of the form $x_i = 1$, there are $\sum_{m=0}^{n-1}\frac{(n-1)!}{m!} \neq \sum_{m=1}^{n-1} \frac{(n-1)!}{m!}$ many vertices that satisfy the equation, hence it also cannot correspond to $\PF_{n-1}$. 
Thus, the only hyperplanes that support a facet of the form $\PF_{n-1}$ are the $n$ mentioned above.
\end{proof}

\subsection{Volume}
Next, we consider a collection of related polytopes, called partial permutahedra, which were first introduced in \cite{HS} in terms of partial permutation matrices; a recursive volume formula was presented in \cite{BCC}.
For our purposes, it suffices to take their vertex description as the definition.

\begin{definition}[Proposition 5.7, \cite{HS}; Proposition 2.6, \cite{BCC}]\label{def:pp-vertices} Let $n$ and $p$ be positive integers. 
The \emph{partial permutahedron} $\mathcal{P}(n,p)$ is the polytope with all permutations of the vectors 
\[(\underbrace{0, \ldots, 0}_{n-k}, \underbrace{p-k+1, \ldots, p-1, p}_k ),\]
for all $0 \leq k \leq \min(n,p)$, as vertices.
\end{definition}

Two integral polytopes $P\subseteq \R^m$ and $Q\subseteq \R^n$ are \emph{integrally equivalent} if there exists an affine transformation $\Phi:\R^m \rightarrow \R^n$ whose restriction to $P$ preserves the lattice.
We establish the following result relating partial permutahedra and classical parking functions.

\begin{proposition}\label{prop:classical-partial}
The classical parking function polytope $\PF_n$ is integrally equivalent to the partial permutahedron $\mathcal{P}(n, n-1)$. In particular, they are related by a translation by the vector $(1,1, \ldots, 1)$.
\end{proposition}

\begin{proof}
Note that all vertices of $\mathcal{P}(n,n-1)$ given in Definition \ref{def:pp-vertices} map to all vertices of $\PF_n$ by the translation $(x_1, x_2, \ldots, x_n) \mapsto (x_1 +1, x_2 +1, \ldots, x_n+1)$.
\end{proof}

We can now establish the equivalence of several different formulas throughout the literature, as they count the normalized volume of the same polytope. 
Previously, \cite{AW} found a generating function and recursive formula for the volume of $\PF_n$, and \cite{Sto} found a closed-form volume formula which uses an alternating sum (inclusion-exclusion). Our new contribution to the following theorem, in part $(iii)$, gives a more simple closed-form volume formula, since it does not contain an alternating sum.

\begin{theorem}\label{thm:main_theorem}
The following are equivalent normalized volume formulas for the classical parking function polytope, where $\nVol(\PF_n) := n!V_n$ denotes the normalized volume:
\begin{enumerate}[(i)]
    \item From \cite{AW}, with $\nVol(\PF_{0})=1$ and $\nVol(\PF_{1})=0$, for $n \geq 2$ we have recursively, \[ \nVol(\PF_n) = (n-1)! \sum_{k=0}^{n-1} \binom{n}{k} \frac{(n-k)^{n-k-1}(n+k-1)}{2} \frac{\nVol(\PF_k)}{k!}.\]
    \item From \cite{BCC}, for $\mathcal{P}(n,n-1)$, with $\nVol(\PF_{0})=1$ and $\nVol(\PF_{1})=0$, for $n \geq 2$ we have recursively,
    \[ \nVol(\PF_n) = (n-1)! \sum_{k=1}^n k^{k-2} \frac{\nVol(\PF_{n-k})}{(n-k)!} \left(k(n-1)- \binom{k}{2}\right) \binom{n}{k}.\]
    \item  \[\nVol(\PF_n) = - \frac{n!}{2^n}\sum_{i=0}^n \binom{n}{i}(2i-3)!!(2n-1)^{n-i}.\]
    \item From \cite{Sto},
    \[\nVol(\PF_n) = n! \sum_{s=1}^n \binom{n-1}{s-1} \frac{n^{n-s}}{2^s} \sum_{i=0}^s (-1)^{s-i}\binom{s}{i}(2i-1)!!. \]
    \item From \cite{She}, equation (33), for $n \geq 2$, \[\nVol(\PF_n)=\frac{n!}{2^n}\sum_{i=0}^{n} (2i-1)(2i-1)!!\binom{n}{i}(2n-1)^{n-i-1}.\]
    \item From \cite{She}, equation (23), for $n \geq 2$, \[\nVol(\PF_n) = n!\frac{n-1}{2^{n-1}}\sum_{i=0}^{n-2}(2i+1)!!\binom{n-2}{i}(2n-1)^{n-i-2}. \]
    \item From \cite{She}, $\nVol(\PF_n)$ equals the number of $n \times n$ $(0,1)$-matrices with two 1's in each row that have positive permanent.
\end{enumerate}
\end{theorem}

\begin{proof}
Proposition \ref{prop:classical-partial} implies $(i) \iff (ii)$.
In the next section, we show Theorem \ref{thm:generalized_closed_form_volume}, which by taking $a=1$ and $b=1$ implies $(i) \iff (iii)$.


From the literature, $(v) \iff (vi) \iff (vii)$ is given in \cite{She}. Equation $(iv)$ is shown to be the normalized volume in \cite{Sto}, and hence is equivalent to $(i)$.

We finish by showing that $(iii) \iff (v)$. This is equivalent to showing that their difference is $0$. From $(v)$, subtract $(iii)$, which gives 
\begin{align*}
    &\frac{n!}{2^n}\sum_{i=0}^{n} (2i-1)(2i-1)!!\binom{n}{i}(2n-1)^{n-i-1}+ \frac{n!}{2^n}\sum_{i=0}^n \binom{n}{i}(2i-3)!!(2n-1)^{n-i}\\
    &= \frac{n!}{2^n}\sum_{i=0}^{n}\binom{n}{i}(2i-3)!!(2n-1)^{n-i-1}[(2i-1)^2 + (2n-1)]\\
    &= \frac{n!}{2^n}\sum_{i=0}^{n}\binom{n}{i}\frac{(2i)!}{2^i i! (2i-1)}(2n-1)^{n-i-1}[(2i-1)^2 + (2n-1)].
\end{align*}
We now use Wilf-Zeilberger theory. 
We need only show that $f(n) := \sum_{i=0}^n F(n,i)$ is $0$, where
\[ F(n,i) := \binom{n}{i}\frac{(2i)!}{2^i i! (2i-1)}(2n-1)^{n-i-1}[(2i-1)^2 + (2n-1)].\]
Note that as $F(n,i)$ contains a binomial coefficient, the sum $f(n)= \sum_{i \in \Z} F(n,i)$.
We use Zeilberger's creative telescoping algorithm \texttt{ct} in the package \texttt{EKHAD} as described in Chapter 6.5 of \cite{PWZ} and available from \cite{Zei}. Calling
\begin{center}
\begin{BVerbatim}
ct(binomial(n,i)*(2*i)!*(2*n-1)^(n-i-1)*
((2*i-1)^2+(2*n-1))/(2^i*i!*(2*i-1)),0,i,n,N);
\end{BVerbatim}
\end{center}
in Maple gives output $1 \cdot f(n) + 0 \cdot f(n+1)= 0$, that is $f(n) = 0$, with certificate $R(n,i) = (-2n+1)i/(2i^2-2i+n)$. If we wish to check our solution, let $G(n,i) := R(n,i)\cdot F(n,i)$. 
Then one can verify that $1 = (G(n,i+1) - G(n,i))/F(n,i)$ via some algebra.
\end{proof}

The sequence for the normalized volume of $\PF_n$ begins \[0, 1, 24, 954, 59040, 5295150, 651354480, 105393619800, 21717404916480, \ldots\] and is OEIS sequence \href{http://oeis.org/A174586}{A174586} \cite{OEIS}.

\section{The convex hull of \texorpdfstring{$\mathbf{x}$}{\textbf{x}}-parking functions}\label{sec:x-parking}

Next, we discuss a generalization of the classical parking functions.
Let $\bx=(x_1,\dots,x_n)\in \Z_{>0}^n$. 
Define an $\bx$\emph{-parking function} to be a sequence $(a_1,\dots,a_n)$ of positive integers whose nondecreasing rearrangement $b_1\leq b_2\leq \cdots \leq b_n$ satisfies $b_i\leq x_1+\cdots + x_i$. 

Throughout Section \ref{sec:x-parking}, we will specialize to vectors of the form $\bx=(a,b,b,\dots,b)$ following the work of Yan \cite{Yan2, Yan}. Additionally the volume calculations in Section \ref{sec:volume} become more difficult for arbitrary $\bx=(x_1,\dots,x_n)$, which we discuss more in Section \ref{sec:general-x}. On the other hand, the vertex and hyperplane descriptions along with some of the basic enumerative results can be easily generalized to the $\bx=(x_1,\dots,x_n)$ setting, and we leave those details to the interested reader.  

As mentioned in \cite{Yan}, from work of Pitman and Stanley \cite{PitmanStanley}, the number of $\bx$-parking functions for $\bx=(a,b,b,\ldots,b)$ is the following:

\begin{theorem}[Theorem 1, \cite{Yan}]\label{thm:number_x-parking}
For $\mathbf{x}=(a,b,b,\ldots,b) \in \Z_{>0}^n$, the number of $\bx$-parking functions is given by $a(a+nb)^{n-1}$.
\end{theorem}

The following definition introduces an  $n$-dimensional polytope associated to $\bx$-parking functions of length $n$ for the specific sequence $\bx=(a,b,b,\dots,b)\in \Z_{>0}^n$, which is one of the main objects of study in this paper. 

\begin{definition}
Define the \emph{$\bx$-parking function polytope} $\Xn$ as the convex hull of all $\bx$-parking functions of length $n$ in $\R^n$ for $\bx=(a,b,b,\dots,b)\in \Z_{>0}^n$. See Figure \ref{fig:x-pfs} for examples. 
\end{definition}

\begin{remark}\label{rem:n=1}
Note that if $n=1$, $\bx = (a)$, so no $b$ is used. 
As a result, we may denote the $\bx$-parking function polytope $\X_{1}(a,b)$ alternatively by $\X_{1}(a)$.
\end{remark}

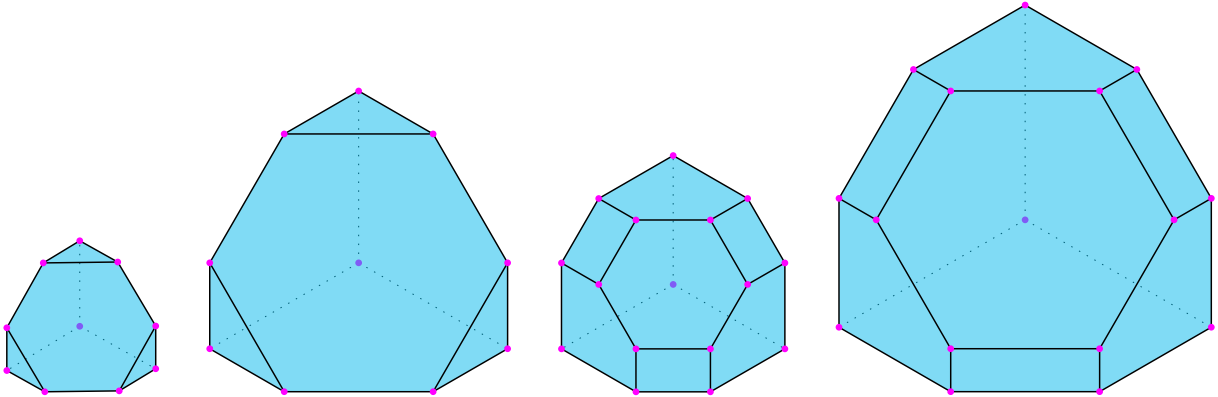
\begin{figure}[h]
    \centering
    \scalebox{.7}{
    \begin{tikzpicture}%
	[x={(-0.692569cm, -0.420596cm)},
	y={(0.721351cm, -0.403777cm)},
	z={(-0.000033cm, 0.812443cm)},
	scale=1.000000,
	back/.style={loosely dotted, thin},
	edge/.style={color=black, thick},
	facet/.style={fill=andresblue,fill opacity=0.500000},
	vertex/.style={inner sep=1pt,circle,draw=andrespink,fill=andrespink,thick}]

\coordinate (1.00000, 1.00000, 1.00000) at (1.00000, 1.00000, 1.00000);
\coordinate (3.00000, 2.00000, 1.00000) at (3.00000, 2.00000, 1.00000);
\coordinate (1.00000, 1.00000, 3.00000) at (1.00000, 1.00000, 3.00000);
\coordinate (3.00000, 1.00000, 2.00000) at (3.00000, 1.00000, 2.00000);
\coordinate (3.00000, 1.00000, 1.00000) at (3.00000, 1.00000, 1.00000);
\coordinate (1.00000, 2.00000, 3.00000) at (1.00000, 2.00000, 3.00000);
\coordinate (1.00000, 3.00000, 1.00000) at (1.00000, 3.00000, 1.00000);
\coordinate (1.00000, 3.00000, 2.00000) at (1.00000, 3.00000, 2.00000);
\coordinate (2.00000, 3.00000, 1.00000) at (2.00000, 3.00000, 1.00000);
\coordinate (2.00000, 1.00000, 3.00000) at (2.00000, 1.00000, 3.00000);
\draw[edge,back] (1.00000, 1.00000, 1.00000) -- (1.00000, 1.00000, 3.00000);
\draw[edge,back] (1.00000, 1.00000, 1.00000) -- (3.00000, 1.00000, 1.00000);
\draw[edge,back] (1.00000, 1.00000, 1.00000) -- (1.00000, 3.00000, 1.00000);
\node[vertex] at (1.00000, 1.00000, 1.00000)     {};
\fill[facet] (1.00000, 3.00000, 2.00000) -- (1.00000, 2.00000, 3.00000) -- (2.00000, 1.00000, 3.00000) -- (3.00000, 1.00000, 2.00000) -- (3.00000, 2.00000, 1.00000) -- (2.00000, 3.00000, 1.00000) -- cycle {};
\fill[facet] (2.00000, 3.00000, 1.00000) -- (1.00000, 3.00000, 1.00000) -- (1.00000, 3.00000, 2.00000) -- cycle {};
\fill[facet] (2.00000, 1.00000, 3.00000) -- (1.00000, 1.00000, 3.00000) -- (1.00000, 2.00000, 3.00000) -- cycle {};
\fill[facet] (3.00000, 1.00000, 1.00000) -- (3.00000, 2.00000, 1.00000) -- (3.00000, 1.00000, 2.00000) -- cycle {};
\draw[edge] (3.00000, 2.00000, 1.00000) -- (3.00000, 1.00000, 2.00000);
\draw[edge] (3.00000, 2.00000, 1.00000) -- (3.00000, 1.00000, 1.00000);
\draw[edge] (3.00000, 2.00000, 1.00000) -- (2.00000, 3.00000, 1.00000);
\draw[edge] (1.00000, 1.00000, 3.00000) -- (1.00000, 2.00000, 3.00000);
\draw[edge] (1.00000, 1.00000, 3.00000) -- (2.00000, 1.00000, 3.00000);
\draw[edge] (3.00000, 1.00000, 2.00000) -- (3.00000, 1.00000, 1.00000);
\draw[edge] (3.00000, 1.00000, 2.00000) -- (2.00000, 1.00000, 3.00000);
\draw[edge] (1.00000, 2.00000, 3.00000) -- (1.00000, 3.00000, 2.00000);
\draw[edge] (1.00000, 2.00000, 3.00000) -- (2.00000, 1.00000, 3.00000);
\draw[edge] (1.00000, 3.00000, 1.00000) -- (1.00000, 3.00000, 2.00000);
\draw[edge] (1.00000, 3.00000, 1.00000) -- (2.00000, 3.00000, 1.00000);
\draw[edge] (1.00000, 3.00000, 2.00000) -- (2.00000, 3.00000, 1.00000);
\node[vertex] at (3.00000, 2.00000, 1.00000)     {};
\node[vertex] at (1.00000, 1.00000, 3.00000)     {};
\node[vertex] at (3.00000, 1.00000, 2.00000)     {};
\node[vertex] at (3.00000, 1.00000, 1.00000)     {};
\node[vertex] at (1.00000, 2.00000, 3.00000)     {};
\node[vertex] at (1.00000, 3.00000, 1.00000)     {};
\node[vertex] at (1.00000, 3.00000, 2.00000)     {};
\node[vertex] at (2.00000, 3.00000, 1.00000)     {};
\node[vertex] at (2.00000, 1.00000, 3.00000)     {};
\end{tikzpicture}
\qquad
\begin{tikzpicture}%
	[x={(-0.707031cm, -0.408259cm)},
	y={(0.707183cm, -0.408200cm)},
	z={(0.000025cm, 0.816516cm)},
	scale=1.000000,
	back/.style={loosely dotted, thin},
	edge/.style={color=black, thick},
	facet/.style={fill=andresblue,fill opacity=0.500000},
	vertex/.style={inner sep=1pt,circle,draw=andrespink,fill=andrespink,thick}]
\coordinate (1.00000, 1.00000, 1.00000) at (1.00000, 1.00000, 1.00000);
\coordinate (5.00000, 3.00000, 1.00000) at (5.00000, 3.00000, 1.00000);
\coordinate (5.00000, 1.00000, 3.00000) at (5.00000, 1.00000, 3.00000);
\coordinate (5.00000, 1.00000, 1.00000) at (5.00000, 1.00000, 1.00000);
\coordinate (1.00000, 1.00000, 5.00000) at (1.00000, 1.00000, 5.00000);
\coordinate (3.00000, 5.00000, 1.00000) at (3.00000, 5.00000, 1.00000);
\coordinate (3.00000, 1.00000, 5.00000) at (3.00000, 1.00000, 5.00000);
\coordinate (1.00000, 5.00000, 3.00000) at (1.00000, 5.00000, 3.00000);
\coordinate (1.00000, 5.00000, 1.00000) at (1.00000, 5.00000, 1.00000);
\coordinate (1.00000, 3.00000, 5.00000) at (1.00000, 3.00000, 5.00000);
\draw[edge,back] (1.00000, 1.00000, 1.00000) -- (5.00000, 1.00000, 1.00000);
\draw[edge,back] (1.00000, 1.00000, 1.00000) -- (1.00000, 1.00000, 5.00000);
\draw[edge,back] (1.00000, 1.00000, 1.00000) -- (1.00000, 5.00000, 1.00000);
\node[vertex] at (1.00000, 1.00000, 1.00000)     {};
\fill[facet] (1.00000, 3.00000, 5.00000) -- (3.00000, 1.00000, 5.00000) -- (5.00000, 1.00000, 3.00000) -- (5.00000, 3.00000, 1.00000) -- (3.00000, 5.00000, 1.00000) -- (1.00000, 5.00000, 3.00000) -- cycle {};
\fill[facet] (5.00000, 1.00000, 1.00000) -- (5.00000, 3.00000, 1.00000) -- (5.00000, 1.00000, 3.00000) -- cycle {};
\fill[facet] (1.00000, 5.00000, 1.00000) -- (3.00000, 5.00000, 1.00000) -- (1.00000, 5.00000, 3.00000) -- cycle {};
\fill[facet] (1.00000, 3.00000, 5.00000) -- (1.00000, 1.00000, 5.00000) -- (3.00000, 1.00000, 5.00000) -- cycle {};
\draw[edge] (5.00000, 3.00000, 1.00000) -- (5.00000, 1.00000, 3.00000);
\draw[edge] (5.00000, 3.00000, 1.00000) -- (5.00000, 1.00000, 1.00000);
\draw[edge] (5.00000, 3.00000, 1.00000) -- (3.00000, 5.00000, 1.00000);
\draw[edge] (5.00000, 1.00000, 3.00000) -- (5.00000, 1.00000, 1.00000);
\draw[edge] (5.00000, 1.00000, 3.00000) -- (3.00000, 1.00000, 5.00000);
\draw[edge] (1.00000, 1.00000, 5.00000) -- (3.00000, 1.00000, 5.00000);
\draw[edge] (1.00000, 1.00000, 5.00000) -- (1.00000, 3.00000, 5.00000);
\draw[edge] (3.00000, 5.00000, 1.00000) -- (1.00000, 5.00000, 3.00000);
\draw[edge] (3.00000, 5.00000, 1.00000) -- (1.00000, 5.00000, 1.00000);
\draw[edge] (3.00000, 1.00000, 5.00000) -- (1.00000, 3.00000, 5.00000);
\draw[edge] (1.00000, 5.00000, 3.00000) -- (1.00000, 5.00000, 1.00000);
\draw[edge] (1.00000, 5.00000, 3.00000) -- (1.00000, 3.00000, 5.00000);
\node[vertex] at (5.00000, 3.00000, 1.00000)     {};
\node[vertex] at (5.00000, 1.00000, 3.00000)     {};
\node[vertex] at (5.00000, 1.00000, 1.00000)     {};
\node[vertex] at (1.00000, 1.00000, 5.00000)     {};
\node[vertex] at (3.00000, 5.00000, 1.00000)     {};
\node[vertex] at (3.00000, 1.00000, 5.00000)     {};
\node[vertex] at (1.00000, 5.00000, 3.00000)     {};
\node[vertex] at (1.00000, 5.00000, 1.00000)     {};
\node[vertex] at (1.00000, 3.00000, 5.00000)     {};
\end{tikzpicture}
\qquad
\begin{tikzpicture}%
	[x={(-0.707031cm, -0.408259cm)},
	y={(0.707183cm, -0.408200cm)},
	z={(0.000025cm, 0.816516cm)},
	scale=1.000000,
	back/.style={loosely dotted, thin},
	edge/.style={color=black, thick},
	facet/.style={fill=andresblue,fill opacity=0.500000},
	vertex/.style={inner sep=1pt,circle,draw=andrespink,fill=andrespink,thick}]
%
%

\coordinate (1.00000, 1.00000, 1.00000) at (1.00000, 1.00000, 1.00000);
\coordinate (4.00000, 3.00000, 2.00000) at (4.00000, 3.00000, 2.00000);
\coordinate (4.00000, 3.00000, 1.00000) at (4.00000, 3.00000, 1.00000);
\coordinate (1.00000, 1.00000, 4.00000) at (1.00000, 1.00000, 4.00000);
\coordinate (4.00000, 2.00000, 3.00000) at (4.00000, 2.00000, 3.00000);
\coordinate (4.00000, 1.00000, 3.00000) at (4.00000, 1.00000, 3.00000);
\coordinate (4.00000, 1.00000, 1.00000) at (4.00000, 1.00000, 1.00000);
\coordinate (3.00000, 4.00000, 2.00000) at (3.00000, 4.00000, 2.00000);
\coordinate (3.00000, 4.00000, 1.00000) at (3.00000, 4.00000, 1.00000);
\coordinate (3.00000, 2.00000, 4.00000) at (3.00000, 2.00000, 4.00000);
\coordinate (3.00000, 1.00000, 4.00000) at (3.00000, 1.00000, 4.00000);
\coordinate (1.00000, 3.00000, 4.00000) at (1.00000, 3.00000, 4.00000);
\coordinate (1.00000, 4.00000, 1.00000) at (1.00000, 4.00000, 1.00000);
\coordinate (2.00000, 4.00000, 3.00000) at (2.00000, 4.00000, 3.00000);
\coordinate (1.00000, 4.00000, 3.00000) at (1.00000, 4.00000, 3.00000);
\coordinate (2.00000, 3.00000, 4.00000) at (2.00000, 3.00000, 4.00000);
\draw[edge,back] (1.00000, 1.00000, 1.00000) -- (1.00000, 1.00000, 4.00000);
\draw[edge,back] (1.00000, 1.00000, 1.00000) -- (4.00000, 1.00000, 1.00000);
\draw[edge,back] (1.00000, 1.00000, 1.00000) -- (1.00000, 4.00000, 1.00000);
\node[vertex] at (1.00000, 1.00000, 1.00000)     {};
\fill[facet] (4.00000, 1.00000, 1.00000) -- (4.00000, 3.00000, 1.00000) -- (4.00000, 3.00000, 2.00000) -- (4.00000, 2.00000, 3.00000) -- (4.00000, 1.00000, 3.00000) -- cycle {};
\fill[facet] (2.00000, 3.00000, 4.00000) -- (3.00000, 2.00000, 4.00000) -- (4.00000, 2.00000, 3.00000) -- (4.00000, 3.00000, 2.00000) -- (3.00000, 4.00000, 2.00000) -- (2.00000, 4.00000, 3.00000) -- cycle {};
\fill[facet] (3.00000, 4.00000, 1.00000) -- (4.00000, 3.00000, 1.00000) -- (4.00000, 3.00000, 2.00000) -- (3.00000, 4.00000, 2.00000) -- cycle {};
\fill[facet] (3.00000, 1.00000, 4.00000) -- (4.00000, 1.00000, 3.00000) -- (4.00000, 2.00000, 3.00000) -- (3.00000, 2.00000, 4.00000) -- cycle {};
\fill[facet] (2.00000, 3.00000, 4.00000) -- (3.00000, 2.00000, 4.00000) -- (3.00000, 1.00000, 4.00000) -- (1.00000, 1.00000, 4.00000) -- (1.00000, 3.00000, 4.00000) -- cycle {};
\fill[facet] (2.00000, 4.00000, 3.00000) -- (2.00000, 3.00000, 4.00000) -- (1.00000, 3.00000, 4.00000) -- (1.00000, 4.00000, 3.00000) -- cycle {};
\fill[facet] (1.00000, 4.00000, 3.00000) -- (1.00000, 4.00000, 1.00000) -- (3.00000, 4.00000, 1.00000) -- (3.00000, 4.00000, 2.00000) -- (2.00000, 4.00000, 3.00000) -- cycle {};
\draw[edge] (4.00000, 3.00000, 2.00000) -- (4.00000, 3.00000, 1.00000);
\draw[edge] (4.00000, 3.00000, 2.00000) -- (4.00000, 2.00000, 3.00000);
\draw[edge] (4.00000, 3.00000, 2.00000) -- (3.00000, 4.00000, 2.00000);
\draw[edge] (4.00000, 3.00000, 1.00000) -- (4.00000, 1.00000, 1.00000);
\draw[edge] (4.00000, 3.00000, 1.00000) -- (3.00000, 4.00000, 1.00000);
\draw[edge] (1.00000, 1.00000, 4.00000) -- (3.00000, 1.00000, 4.00000);
\draw[edge] (1.00000, 1.00000, 4.00000) -- (1.00000, 3.00000, 4.00000);
\draw[edge] (4.00000, 2.00000, 3.00000) -- (4.00000, 1.00000, 3.00000);
\draw[edge] (4.00000, 2.00000, 3.00000) -- (3.00000, 2.00000, 4.00000);
\draw[edge] (4.00000, 1.00000, 3.00000) -- (4.00000, 1.00000, 1.00000);
\draw[edge] (4.00000, 1.00000, 3.00000) -- (3.00000, 1.00000, 4.00000);
\draw[edge] (3.00000, 4.00000, 2.00000) -- (3.00000, 4.00000, 1.00000);
\draw[edge] (3.00000, 4.00000, 2.00000) -- (2.00000, 4.00000, 3.00000);
\draw[edge] (3.00000, 4.00000, 1.00000) -- (1.00000, 4.00000, 1.00000);
\draw[edge] (3.00000, 2.00000, 4.00000) -- (3.00000, 1.00000, 4.00000);
\draw[edge] (3.00000, 2.00000, 4.00000) -- (2.00000, 3.00000, 4.00000);
\draw[edge] (1.00000, 3.00000, 4.00000) -- (1.00000, 4.00000, 3.00000);
\draw[edge] (1.00000, 3.00000, 4.00000) -- (2.00000, 3.00000, 4.00000);
\draw[edge] (1.00000, 4.00000, 1.00000) -- (1.00000, 4.00000, 3.00000);
\draw[edge] (2.00000, 4.00000, 3.00000) -- (1.00000, 4.00000, 3.00000);
\draw[edge] (2.00000, 4.00000, 3.00000) -- (2.00000, 3.00000, 4.00000);
\node[vertex] at (4.00000, 3.00000, 2.00000)     {};
\node[vertex] at (4.00000, 3.00000, 1.00000)     {};
\node[vertex] at (1.00000, 1.00000, 4.00000)     {};
\node[vertex] at (4.00000, 2.00000, 3.00000)     {};
\node[vertex] at (4.00000, 1.00000, 3.00000)     {};
\node[vertex] at (4.00000, 1.00000, 1.00000)     {};
\node[vertex] at (3.00000, 4.00000, 2.00000)     {};
\node[vertex] at (3.00000, 4.00000, 1.00000)     {};
\node[vertex] at (3.00000, 2.00000, 4.00000)     {};
\node[vertex] at (3.00000, 1.00000, 4.00000)     {};
\node[vertex] at (1.00000, 3.00000, 4.00000)     {};
\node[vertex] at (1.00000, 4.00000, 1.00000)     {};
\node[vertex] at (2.00000, 4.00000, 3.00000)     {};
\node[vertex] at (1.00000, 4.00000, 3.00000)     {};
\node[vertex] at (2.00000, 3.00000, 4.00000)     {};
\end{tikzpicture}
\qquad
\begin{tikzpicture}%
	[x={(-0.707031cm, -0.408259cm)},
	y={(0.707183cm, -0.408200cm)},
	z={(0.000025cm, 0.816516cm)},
	scale=1.000000,
	back/.style={loosely dotted, thin},
	edge/.style={color=black, thick},
	facet/.style={fill=andresblue,fill opacity=0.500000},
	vertex/.style={inner sep=1pt,circle,draw=andrespink,fill=andrespink,thick}]
\coordinate (1.00000, 1.00000, 1.00000) at (1.00000, 1.00000, 1.00000);
\coordinate (6.00000, 4.00000, 2.00000) at (6.00000, 4.00000, 2.00000);
\coordinate (6.00000, 4.00000, 1.00000) at (6.00000, 4.00000, 1.00000);
\coordinate (6.00000, 2.00000, 4.00000) at (6.00000, 2.00000, 4.00000);
\coordinate (6.00000, 1.00000, 4.00000) at (6.00000, 1.00000, 4.00000);
\coordinate (1.00000, 1.00000, 6.00000) at (1.00000, 1.00000, 6.00000);
\coordinate (6.00000, 1.00000, 1.00000) at (6.00000, 1.00000, 1.00000);
\coordinate (4.00000, 6.00000, 2.00000) at (4.00000, 6.00000, 2.00000);
\coordinate (4.00000, 6.00000, 1.00000) at (4.00000, 6.00000, 1.00000);
\coordinate (4.00000, 2.00000, 6.00000) at (4.00000, 2.00000, 6.00000);
\coordinate (4.00000, 1.00000, 6.00000) at (4.00000, 1.00000, 6.00000);
\coordinate (2.00000, 6.00000, 4.00000) at (2.00000, 6.00000, 4.00000);
\coordinate (2.00000, 4.00000, 6.00000) at (2.00000, 4.00000, 6.00000);
\coordinate (1.00000, 6.00000, 4.00000) at (1.00000, 6.00000, 4.00000);
\coordinate (1.00000, 6.00000, 1.00000) at (1.00000, 6.00000, 1.00000);
\coordinate (1.00000, 4.00000, 6.00000) at (1.00000, 4.00000, 6.00000);
\draw[edge,back] (1.00000, 1.00000, 1.00000) -- (1.00000, 1.00000, 6.00000);
\draw[edge,back] (1.00000, 1.00000, 1.00000) -- (6.00000, 1.00000, 1.00000);
\draw[edge,back] (1.00000, 1.00000, 1.00000) -- (1.00000, 6.00000, 1.00000);
\node[vertex] at (1.00000, 1.00000, 1.00000)     {};
\fill[facet] (6.00000, 1.00000, 1.00000) -- (6.00000, 4.00000, 1.00000) -- (6.00000, 4.00000, 2.00000) -- (6.00000, 2.00000, 4.00000) -- (6.00000, 1.00000, 4.00000) -- cycle {};
\fill[facet] (2.00000, 4.00000, 6.00000) -- (4.00000, 2.00000, 6.00000) -- (6.00000, 2.00000, 4.00000) -- (6.00000, 4.00000, 2.00000) -- (4.00000, 6.00000, 2.00000) -- (2.00000, 6.00000, 4.00000) -- cycle {};
\fill[facet] (4.00000, 6.00000, 1.00000) -- (6.00000, 4.00000, 1.00000) -- (6.00000, 4.00000, 2.00000) -- (4.00000, 6.00000, 2.00000) -- cycle {};
\fill[facet] (1.00000, 6.00000, 1.00000) -- (4.00000, 6.00000, 1.00000) -- (4.00000, 6.00000, 2.00000) -- (2.00000, 6.00000, 4.00000) -- (1.00000, 6.00000, 4.00000) -- cycle {};
\fill[facet] (1.00000, 4.00000, 6.00000) -- (2.00000, 4.00000, 6.00000) -- (2.00000, 6.00000, 4.00000) -- (1.00000, 6.00000, 4.00000) -- cycle {};
\fill[facet] (1.00000, 4.00000, 6.00000) -- (1.00000, 1.00000, 6.00000) -- (4.00000, 1.00000, 6.00000) -- (4.00000, 2.00000, 6.00000) -- (2.00000, 4.00000, 6.00000) -- cycle {};
\fill[facet] (4.00000, 1.00000, 6.00000) -- (6.00000, 1.00000, 4.00000) -- (6.00000, 2.00000, 4.00000) -- (4.00000, 2.00000, 6.00000) -- cycle {};
\draw[edge] (6.00000, 4.00000, 2.00000) -- (6.00000, 4.00000, 1.00000);
\draw[edge] (6.00000, 4.00000, 2.00000) -- (6.00000, 2.00000, 4.00000);
\draw[edge] (6.00000, 4.00000, 2.00000) -- (4.00000, 6.00000, 2.00000);
\draw[edge] (6.00000, 4.00000, 1.00000) -- (6.00000, 1.00000, 1.00000);
\draw[edge] (6.00000, 4.00000, 1.00000) -- (4.00000, 6.00000, 1.00000);
\draw[edge] (6.00000, 2.00000, 4.00000) -- (6.00000, 1.00000, 4.00000);
\draw[edge] (6.00000, 2.00000, 4.00000) -- (4.00000, 2.00000, 6.00000);
\draw[edge] (6.00000, 1.00000, 4.00000) -- (6.00000, 1.00000, 1.00000);
\draw[edge] (6.00000, 1.00000, 4.00000) -- (4.00000, 1.00000, 6.00000);
\draw[edge] (1.00000, 1.00000, 6.00000) -- (4.00000, 1.00000, 6.00000);
\draw[edge] (1.00000, 1.00000, 6.00000) -- (1.00000, 4.00000, 6.00000);
\draw[edge] (4.00000, 6.00000, 2.00000) -- (4.00000, 6.00000, 1.00000);
\draw[edge] (4.00000, 6.00000, 2.00000) -- (2.00000, 6.00000, 4.00000);
\draw[edge] (4.00000, 6.00000, 1.00000) -- (1.00000, 6.00000, 1.00000);
\draw[edge] (4.00000, 2.00000, 6.00000) -- (4.00000, 1.00000, 6.00000);
\draw[edge] (4.00000, 2.00000, 6.00000) -- (2.00000, 4.00000, 6.00000);
\draw[edge] (2.00000, 6.00000, 4.00000) -- (2.00000, 4.00000, 6.00000);
\draw[edge] (2.00000, 6.00000, 4.00000) -- (1.00000, 6.00000, 4.00000);
\draw[edge] (2.00000, 4.00000, 6.00000) -- (1.00000, 4.00000, 6.00000);
\draw[edge] (1.00000, 6.00000, 4.00000) -- (1.00000, 6.00000, 1.00000);
\draw[edge] (1.00000, 6.00000, 4.00000) -- (1.00000, 4.00000, 6.00000);
\node[vertex] at (6.00000, 4.00000, 2.00000)     {};
\node[vertex] at (6.00000, 4.00000, 1.00000)     {};
\node[vertex] at (6.00000, 2.00000, 4.00000)     {};
\node[vertex] at (6.00000, 1.00000, 4.00000)     {};
\node[vertex] at (1.00000, 1.00000, 6.00000)     {};
\node[vertex] at (6.00000, 1.00000, 1.00000)     {};
\node[vertex] at (4.00000, 6.00000, 2.00000)     {};
\node[vertex] at (4.00000, 6.00000, 1.00000)     {};
\node[vertex] at (4.00000, 2.00000, 6.00000)     {};
\node[vertex] at (4.00000, 1.00000, 6.00000)     {};
\node[vertex] at (2.00000, 6.00000, 4.00000)     {};
\node[vertex] at (2.00000, 4.00000, 6.00000)     {};
\node[vertex] at (1.00000, 6.00000, 4.00000)     {};
\node[vertex] at (1.00000, 6.00000, 1.00000)     {};
\node[vertex] at (1.00000, 4.00000, 6.00000)     {};
\end{tikzpicture}
}
    \caption{The $\mathbf{x}$-parking function polytopes, from left to right: $\mathfrak{X}_3(1,1)$, $\mathfrak{X}_3(1,2)$, $\mathfrak{X}_3(2,1)$, and $\mathfrak{X}_3(2,2)$. Observe that $\mathfrak{X}_3(1,2)$ is a dilate of $\mathfrak{X}_3(1,1)$. Note that when $a>1$, there are new facets that do not appear when $a=1$.}
    \label{fig:x-pfs}
\end{figure}


\subsection{Face structure}

In this subsection we describe the face structure of the \emph{$\bx$-parking function polytope} $\Xn$.

From Theorem \ref{thm:number_x-parking}, we obtain an upper bound for the number of $\bx$-parking functions which arise as vertices of $\Xn$ since it is well-known that if a polytope can be written as the convex hull of a finite set of points, then the set contains all the vertices of the polytope (Proposition 2.2, \cite{Ziegler}). 
We now give a vertex description of $\Xn$.

\begin{proposition}\label{prop:abn_vertex}
The vertices of $\Xn$ are all permutations of 
\[ (\underbrace{1, \ldots, 1}_k, \underbrace{a+kb, a+(k+1)b, \ldots, a+(n-2)b, a+(n-1)b}_{n-k}),\] for all $0 \leq k \leq n$.
\end{proposition}

\begin{proof}
Consider an $\bx$-parking function $x = (x_1, \ldots, x_n)$ for which there is a term $x_i > 1$ such that $(x_1, \ldots, x_{i-1}, x_{i} + 1, x_{i+1}, \ldots, x_n)$ is also an $\bx$-parking function. 
Then $x$ is a convex combination of two other $\bx$-parking functions. 
Second, if \[ x= (\underbrace{1, \ldots, 1}_k, \underbrace{a+kb, a+(k+1)b, \ldots, a+(n-2)b, a+(n-1)b}_{n-k})\] is a convex combination of $y,z \in \Xn$, then $x=y=z$, as the first $k$ coordinates of $x$ are minimal at 1 and the last $(n-k)$ coordinates are maximized due to the condition on the nondecreasing rearrangement of $\bx$-parking functions.
Thus, $x$ is a vertex of $\Xn$.
\end{proof}

Next, we enumerate the vertices of $\Xn$. 
In the case that $a=1$, we have the same number of vertices as in the case of the classical parking function polytope $\PF_n$, as $b$ is a parameter that increases the lengths of the edges of the polytope.
However, when $a >1$, new vertices and edges come into play.

\begin{proposition}\label{thm:num_vertices}
The number of vertices of $\Xn$ is 
\[ \begin{cases} 
n!\left( \frac{1}{1!} + \cdots + \frac{1}{n!} \right) & \text{if } a = 1\\
n!\left( \frac{1}{0!} + \frac{1}{1!} + \cdots + \frac{1}{n!} \right) & \text{if } a > 1.
\end{cases}\]
\end{proposition}

\begin{proof}
The vertices are the permutations of the following: 
\begin{align*}
    v_0 &= (1,\ldots,1),\\
    v_1 &= (1,\ldots,1, a+(n-1)b),\\
    v_2 &= (1,\ldots,1, a+(n-2)b, a+(n-1)b),\\
    &\vdots\\
    v_{n-1} &= (1,a+b,\ldots, a+(n-1)b),\\
    v_{n} &= (a,a+b,\ldots, a+(n-1)b).
\end{align*}
Observe that there is one permutation of $v_0$, $n$ permutations of $v_1$, $n(n-1)$ permutations of $v_2$, and in general, $n(n-1)\cdots(n-k+1)$ permutations of $v_k$. 
If $a>1$, then the vertices $v_{n-1}$ and $v_n$ are distinct, so we count the contribution of both. 
However, if $a=1$, then $v_{n-1}=v_n$, and we only count one. 
For $a>1$, this gives
\begin{align*}
    &1 + n+ n(n-1) + \cdots + \left(n(n-1)\cdots 2\right) + \left(n(n-1)\cdots 2\cdot1\right)\\ &= n!\left(\frac{1}{n!} + \frac{1}{(n-1)!} + \frac{1}{(n-2)!} + \cdots + \frac{1}{1!} + \frac{1}{0!}\right),
\end{align*} 
and for $a=1$,
\begin{align*}
    &1 + n+ n(n-1) + \cdots + \left(n(n-1)\cdots 2\right)\\ &= n!\left(\frac{1}{n!} + \frac{1}{(n-1)!} + \frac{1}{(n-2)!} + \cdots + \frac{1}{1!}\right).
\end{align*}
\end{proof}

We continue by presenting an inequality description of $\Xn$.
\begin{proposition}\label{prop:inequality}
The $\bx$-parking function polytope $\Xn$ is given by the following minimal inequality description:
\begin{align*}
\intertext{\emph{For all} $1 \leq i \leq n$,}
    1 &\leq x_i \leq (n-1)b+a,  \\
\intertext{\emph{for all} $1 \leq i < j \leq n$,}    
    x_i+x_j &\leq \left((n-2)b+a\right) + \left( (n-1)b+a\right),\\
\intertext{\emph{for all} $1 \leq i < j < k \leq n$,}    
    x_i+x_j+x_k &\leq \left((n-3)b+a\right) + \left((n-2)b+a\right) + \left( (n-1)b+a\right),\\
    &\vdots\\
\intertext{\emph{for all} $1 \leq i_1 < i_2 < \cdots < i_{n-2} \leq n$,}     
    x_{i_1}+ x_{i_2}+ \cdots + x_{i_{n-2}} &\leq (2b+a) + \cdots + \left( (n-2)b+a\right) + \left((n-1)b+a\right),\\
\intertext{\emph{if $a > 1$, for all} $1 \leq i_1 < i_2 < \cdots < i_{n-1} \leq n$,}
    x_{i_1}+ x_{i_2}+ \cdots + x_{i_{n-1}} &\leq (b+a) + \cdots + \left( (n-2)b+a\right) + \left((n-1)b+a\right),\\
\intertext{\emph{and (regardless of $a$),}}    
    x_{1}+ x_{2}+ \cdots + x_{n} &\leq a + \cdots + \left( (n-2)b+a\right) + \left((n-1)b+a\right).
\end{align*}
\end{proposition}

\begin{proof}
We use the vertex description given in Proposition \ref{prop:abn_vertex}, which states that the vertices of $\Xn$ are all permutations of 
\[ (\underbrace{1, \ldots, 1}_k, \underbrace{a+kb, a+(k+1)b, \ldots, a+(n-2)b, a+(n-1)b}_{n-k})\] for all $0 \leq k \leq n$.

First consider that since $a,b \geq 1$, all coordinates in a vertex of $\Xn$ are at least $1$, i.e., $1 \leq x_i$ for all $1 \leq i \leq n$. 
Now, we turn our attention to inequalities that are solely upper bounds, so we can assume that $k=0$ and only consider the largest possible coordinates. 
The largest a single coordinate could be is $a+(n-1)b$, so $x_i \leq a+(n-1)b$ for all $1\leq i \leq n$.

Next, summing the largest two coordinates is at most $\left((n-2)b+a\right) + \left( (n-1)b+a\right)$, so \[x_i + x_j \leq \left((n-2)b+a\right) + \left( (n-1)b+a\right), \text{ for all } 1 \leq i < j \leq n.\]

Repeating the same process with summing the largest possible $3, \ldots, n$ variables, completes the inequality description.

Note that in the case that $a=1$, the inequality that bounds the sum of any $n-1$ variables is not needed in the minimal inequality description.
We justify this as follows.
Given that $a=1$, the inequality is that \[x_{i_1}+ x_{i_2}+ \cdots + x_{i_{n-1}} \leq \frac{n(n-1)}{2} b +n-1 \text{ for all } 1 \leq i_1 < i_2 < \cdots < i_{n-1} \leq n.\] 
Also, we always need the final inequality $x_{1}+ x_{2}+ \cdots + x_{n} \leq \frac{n(n-1)}{2} b +n$. 
If we sum the $n-1$ largest terms of \[ (a, a+b, a+2b, \ldots, a+(n-2)b, a+(n-1)b),\] the only value that is not used is $a$, which in this case equals $1$. 
Hence, the final inequality \[x_{1}+ x_{2}+ \cdots + x_{n} \leq \frac{n(n-1)}{2} b +n\] is sufficient, because removing one coordinate from the sum on its left hand side subtracts $a=1$ from the right hand side, exactly yielding the inequality \[x_{i_1}+ x_{i_2}+ \cdots + x_{i_{n-1}} \leq \frac{n(n-1)}{2} b +n-1.\] 
Note that this only happens in the case that $a=1$ because if $a>1$, we are able to get a tighter inequality when we subtract $a >1$ from the right hand side as before.

We now show that the inequality description is minimal.
Observe that by construction for each of these defining inequalities, there exists an extremal point of the polytope which gives equality.
It is clear that the inequalities $1\leq x_i$ in each coordinate are necessary, as they are the only lower bounds on the coordinates. 
We only need to consider the minimality of the set of inequalities which are upper bounds on partial sums of the coordinates.
A general inequality on $1 \leq k \leq n$ variables says
\begin{align*}
    x_{i_1} + x_{i_2} + \cdots + x_{i_k} &\leq b\left(nk - \frac{k^2+k}{2}\right) + ka
\end{align*}
for some variables $1 \leq i_1 < i_2 < \cdots < i_k \leq n$.
Observe that for each of the inequalities on $k$ variables (and of course excluding the case of $a=1$ and $k=n-1$), it is strictly stronger on at least one point than all inequalities on $1, \ldots, k-1$ variables. 
Explicitly, a general inequality on $1 \leq k-1 \leq n$ variables says
\begin{align}
    x_{i_1} + x_{i_2} + \cdots + x_{i_{k-1}} &\leq b\left(n(k-1) - \frac{(k-1)^2+k-1}{2}\right) + (k-1)a\label{eq:proof_k-1}\\
    &= b\left(nk - \frac{k^2+k}{2} +(k-n)\right) + ka -a\nonumber\\
    &= b\left(nk - \frac{k^2+k}{2}\right) + ka +[b(k-n) -a].\nonumber
\end{align}
Note that adding one coordinate $x_k$ to the left hand side of (\ref{eq:proof_k-1}) and $a+(n-1)b$ to the right hand side gives
\begin{align*} x_{i_1} + x_{i_2} + \cdots + x_{i_k} &\leq b\left(nk - \frac{k^2+k}{2}\right) + ka +[b(k-n) -a] + a+(n-1)b\\
&= b\left(nk - \frac{k^2+k}{2}\right) + ka +b(k-1),
\end{align*}
which is a worse bound for all $k > 1$ due to the $b(k-1)$ term.
As we constructed this inequality description in increasing order on the number of variables, each one refining the previous, the inequality description is an irredundant system of inequalities and is therefore minimal.
\end{proof}

Since we have a minimal inequality description of $\Xn$, by enumerating the number of defining inequalities we obtain a count for the number of facets of $\Xn$.

\begin{corollary}\label{thm:x-facets}
The number of facets of $\Xn$ is $2^n-1$ if $a=1$ and $2^n-1+n$ if $a > 1$.
\end{corollary}

\begin{proof} 
It suffices to count the number of minimal defining inequalities given in Proposition \ref{prop:inequality}. Note that there are $\binom{n}{k}$ different inequalities with $1 < k < n-1$ or $k=n$ distinct variables. We have $2\binom{n}{1}$ many inequalities with one variable since we have $1\leq x_i$ and $x_i \leq (n-1)b + a$.
If $a=1$, there are no inequalities in $n-1$ variables. Then the number of inequalities is
\[ \binom{n}{1} + \binom{n}{1} + \binom{n}{2} + \cdots + \binom{n}{n-2} + \binom{n}{n} = 2^n-1.\]
If $a > 1$, then there are $\binom{n}{n-1}$ inequalities in $n-1$ variables, so the number of inequalities is
\[ \binom{n}{1} +  \left[\binom{n}{1} + \binom{n}{2} + \cdots + \binom{n}{n-2} + \binom{n}{n-1}+ \binom{n}{n} \right] = n + 2^n-1.\]
\end{proof}

Next we study the edges of $\Xn$.

\begin{definition}
Let $x$ be a vertex of $\Xn$.
Then it is a permutation of 
\[(1, \ldots, 1, a+kb, a+(k+1)b, \ldots, a+(n-2)b, a+(n-1)b),\]
for some unique $1 \leq k \leq n$ if $a=1$ or $0 \leq k \leq n$ if $a>1$. We say that $x$ is on \textit{layer} $n-k$. For $x = (1, \ldots, 1)$, we say that it is on layer 0.
\end{definition}

\begin{lemma}
If $v,u$ are two vertices of $\Xn$ such that $vu$ is an edge, then the layers of $v$ and $u$ differ by at most 1.
\end{lemma}

\begin{proof}
The proof follows almost identically to that of Proposition 2.1 in \cite{AW} in the case that $a=1$.
In the case that $a >1$, there is just one more layer to address.

Let $c \cdot x = c_1x_1 + \cdots + c_nx_n$ be the dot product of vectors $c,x \in \R^n$. 
If $vu$ is an edge, then there exists $c \in \R^n$ such that $c \cdot v = c \cdot u > c \cdot w$ for any vertex $w$ of $\Xn$ such that $w \neq v,u$.
Note that $\Xn$ is invariant under permutation of the coordinates, so without loss of generality, assume that $c_1 \leq c_2 \leq \cdots \leq c_n$. 

Suppose for the sake of contradiction that $v,u$ are $t \geq 2$ layers apart.
Without loss of generality, suppose that $u$ is below $v$. 
Hence, $v$ is (up to permutation)
$$(1, \ldots, 1, a+kb, a+(k+1)b, \ldots, a+(n-2)b, a+(n-1)b),$$ 
and $u$ is (up to permutation) 
$$(1, \ldots, 1, a+(k+t)b, a+(k+t+1)b, \ldots, a+(n-2)b, a+(n-1)b),$$
where $1 \leq k < k+2 \leq k+t \leq n$ if $a=1$ and $0 \leq k < k+2 \leq k+t \leq n$ if $a > 1$.
Specifically, since $v,u$ uniquely maximizes $c \cdot x$, then by the rearrangement inequality, $v$ and $u$ are exactly 
$$(1, \ldots, 1, a+kb, a+(k+1)b, \ldots, a+(n-2)b, a+(n-1)b),$$ and 
$$(1, \ldots, 1, a+(k+t)b, a+(k+t+1)b, \ldots, a+(n-2)b, a+(n-1)b),$$
respectively (not just up to permutation), and $c_{k-1} < c_k < \cdots < c_n$. 
Consider two cases. 
First, if $c_{k+t-1} \geq 0$, then for $w = (1, \ldots, 1, a+(k+t-1)b, a+(k+t)b, \ldots, a+(n-2)b, a+(n-1)b) \in \Xn$ which is not equal to $v,u$, we have that $c \cdot w \geq c \cdot u$, a contradiction. Otherwise, if $c_{k+t-1}<0$, we have $c_k < \cdots < c_{k+t-1}< 0$, so $c \cdot v - c \cdot u = c_k(a+kb - 1) + c_{k+1}(a+(k+1)b -1) + \cdots + c_{k+t-1}(a+(k+ t -1)b -1) <0$. This implies $c \cdot v < c \cdot u$, a contradiction.
\end{proof}

\begin{remark}
For fixed $n$, if $a=1$, then regardless of the value for $b$, there are the same number of layers.
If $a>1$, then regardless of the value for $b$, there are the same number of layers, which is one more than in the case where $a=1$.
\end{remark}

\begin{proposition}\label{2.2}
For each vertex $v$ of $\Xn$, there are exactly $n$ edges of $\Xn$ with $v$ as one of the vertices. 
That is, $\Xn$ is a simple polytope.
\end{proposition}

\begin{proof}[Proof sketch]
In the case that $a=1$, the proof is identical to the proof of Proposition 2.2 in \cite{AW}, where all instances of coordinates $(1, \ldots, 1, k+1, k+2, \ldots, n-1, n)$ are replaced by $(1, \ldots, 1, 1+kb, 1+(k+1)b, \ldots, 1+(n-2)b, 1+(n-1)b)$. The same proof works because $a=1$, so no new layers are introduced.

In the case that $a >1$, there is one more layer to address, as $(1, a+b, a+2b, \ldots, a+(n-2)b, a+(n-1)b) \neq (a, a+b, a+2b, \ldots, a+(n-2)b, a+(n-1)b)$ are distinct now. By modifying the proof from the case of $a=1$, one only needs to check that layers $n-2$, $n-1$ (the new layer, meaning the layer that causes new facets to appear), and $n$ have vertices with the correct number of edges.
\end{proof}

A polytope is \textit{smooth} if it is simple and at each vertex $v$, the primitive edge directions from $v$ form a lattice basis of $\mathbb{Z}^n$. More background on smooth polytopes, as well as their connections to toric ideals and projective geometry, can be found in \cite{CLS}.
\begin{proposition}
     For any choice of positive integers $a,b,n$, the polytope $\mathfrak{X}_{n}(a,b)$ is smooth.
\end{proposition}

\begin{proof}
By symmetry, it suffices to check this at some vertex $v$ on each layer $n-k$ for $0\leq k\leq n$ (or $1 \leq  k\leq n$ in the case where $a=1$). 
Without loss of generality, let $v = (1,\dots 1, a+kb,\dots, a+(n-1)b)$.
From the proof of Proposition \ref{2.2}, for any given vertex $v$ of $\mathfrak{X}_{n}(a,b)$, we have determined the $n$ vertices $u_1,\dots, u_n$ such that $vu_i$ is an edge. 

If $k=n$, we have $v = (1,\dots, 1)$ and the $u_i$ are all the permutations of $(1,\dots, 1, n)$. 
Thus, all the primitive edge directions, determined by $v$ and the first lattice point along each edge containing it, clearly form an integral basis of $\mathbb{Z}^n$.

Now, assume $a>1$ and consider $2\leq k \leq n-1$. 
Let $u_1, \dots, u_k$ be on the layer above $v$, where $u_i$ is obtained by changing the $i$-th coordinate of $v$ from 1 to $a+(k-1)b$. 
Let $u_{k+1}, \dots, u_{n-1}$ be on the same layer as $v$, where $u_j$ is obtained by switching the $j$-th and $(j+1)$-st coordinates of $v$. 
The last vertex $u_n$ is on the layer below $v$ and is obtained by changing the $k$-th coordinate of $v$ from $a+kb$ to 1.
We can see that the primitive edge directions from $v$ to $u_1,\dots, u_n$ form a basis of $\mathbb{Z}^n$, where the first $k$ vectors are basis for the first $k$-coordinates of $\mathbb{Z}^n$ and the last $(n-k)$ are a basis of the last $n-k$ coordinates.

If $k =1$, we have $v = (1, a+b, \dots, a+(n-1)b)$. 
Let $u_1 = (a, a+b,\dots, a+(n-1)b)$, which is on the layer above $v$.
Let $u_2,\dots, u_{n-1}$ be on the same layer as $v$ where $u_j$ is obtained by switching the $j$-th and $(j+1)$-th coordinates of $v$. 
Finally, let $u_n = (1,1, a+2b,\dots, a+(n-1)b)$, which is on the layer below $v$. 
It is clear that the primitive edge directions from $v$ to these $u_i$ form an integral basis of $\mathbb{Z}^n$.

If $k = 0$, we have $v  = (a, a+b,\dots, a+(n-1)b)$.
Let $u_1,\dots, u_{n-1}$ be on the same layer as $v$ where $u_j$  is obtained by switching the $j$-th and $(j+1)$-st coordinates of $v$.
Finally, let $u_n = (1,a+b, a+2b,\dots, a+(n-1)b)$, which is on the layer below $v$.
It is clear that the primitive edge directions from $v$ to these $u_i$ form an integral basis of $\mathbb{Z}^n$.

A similar proof holds for $a=1$: note that when $a=1$ we have one less layer.
\end{proof}

\begin{proposition}
The number of edges of $\Xn$ is $\frac{n}{2}\Ver(\Xn)$, where $\Ver(\Xn)$ denotes the number of vertices $\Xn$.
\end{proposition}

\begin{proof}
By Proposition \ref{2.2}, the graph of $\Xn$ is an $n$-regular graph. This gives the desired formula, where the number of vertices is given by Proposition \ref{thm:num_vertices}.
\end{proof}

Up to this point we have the number of $0$-dimensional faces (vertices), the $1$-dimensional faces (edges), and $(n-1)$-dimensional facets (facets).
In what comes next, we study the  faces of higher dimension.

\begin{proposition}\label{thm:faces}
Let $f_{k}$ be the number of $k$-dimensional faces of $\Xn$ for $k \in \{ 0, \ldots, n\}$. Then if $a=1$,
\[f_{k} = \sum_{m=0,\ m \neq 1}^{n-k} \binom{n}{m} \cdot (n-k-m)! \cdot S(n-m+1, n-k-m+1),\] and if $a > 1$,
\[f_{k} = \sum_{m=0}^{n-k} \binom{n}{m} \cdot (n-k-m)! \cdot S(n-m+1, n-k-m+1),\] where $S(n,k)$ are the Stirling numbers of the second kind.   
\end{proposition}

\begin{proof}[Proof sketch]
In the case that $a=1$, the proof is identical to the proof of Theorem 3.1 in \cite{AW}, where all instances of coordinates $(1, \ldots, 1, k+1, k+2, \ldots, n-1, n)$ are replaced by $(1, \ldots, 1, 1+kb, 1+(k+1)b, \ldots, 1+(n-2)b, 1+(n-1)b)$. The same proof works because $a=1$, so no new layers are introduced, and thus no new faces are introduced.

In the case that $a >1$, there is one more layer to address, as $(1, a+b, a+2b, \ldots, a+(n-2)b, a+(n-1)b) \neq (a, a+b, a+2b, \ldots, a+(n-2)b, a+(n-1)b)$ are distinct now. By modifying the proof from the case of $a=1$, the restriction given in Lemma 3.1 in \cite{AW} which causes $m = 1$ to be excluded from the sum in the proof of Theorem 3.1 in \cite{AW} is removed due to the additional layer.
\end{proof}

The faces of a convex polytope $P$ form a lattice called its \emph{face lattice}, denoted by $\mathcal{F}(P)$, where the partial ordering is given by set inclusion of the faces.
Two polytopes whose face lattices are isomorphic (as unlabeled partially ordered sets) are \emph{combinatorially equivalent}.
We continue with a characterization of which $\Xn$ are combinatorially equivalent.

\begin{proposition}
 For fixed $n$ and for $a=1$, the $\mathfrak{X}_{n}(1,b)$ are combinatorially equivalent for all $b \geq 1$.
Additionally, for fixed $n$, the $\Xn$ are combinatorially equivalent for all $a > 1$ and $b \geq 1$.
\end{proposition}

\begin{proof}
By Proposition \ref{thm:faces}, for each $b$, the face lattice of $\mathfrak{X}_{n}(1,b)$ has the same number of elements, and the same number of elements at each level corresponding to the dimension of the face.
For a face lattice, the join of two elements in the same layer will be in the layer above it, and the meet of two elements in the same layer will be in the layer below it, as face inclusions pass through faces of one dimension more/less (face inclusions of faces of the same dimension result in equality).
Thus, it only remains to construct the isomorphism.
Define $\varphi: \mathcal{F}(\mathfrak{X}_{n}(1,b_1)) \to \mathcal{F}(\mathfrak{X}_{n}(1,b_2))$ by replacing $b_1$ by $b_2$ in the vertex description of each face in the face lattice.
It is clear that this is an isomorphism, as the inverse is given by replacing $b_2$ by $b_1$.
Therefore, the $\mathfrak{X}_{n}(1,b)$ are combinatorially equivalent for $a=1$ and for any $b\geq 1$. 

Similarly as before, for all $a>1$, we define $\varphi: \mathcal{F}(\mathfrak{X}_{n}(a_1,b_1)) \to \mathcal{F}(\mathfrak{X}_{n}(a_2,b_2))$ by replacing $a_1$ by $a_2$ and $b_1$ by $b_2$ in the vertex description of each face in the face lattice.
Hence, the $\Xn$ are combinatorially equivalent for any $a>1$ and for all $b\geq 1$. 
\end{proof}


\subsection{Volume}\label{sec:volume}
We turn our attention to calculating the volume of  $\bx$-parking function polytopes $\Xn$.
We extend previous work of \cite{AW} to find a recursive volume formula in Theorem \ref{thm:generalized_recursive_volume}. 
Then, by using exponential generating functions, we find a closed-form expression for the volume of  $\Xn$ in Theorem \ref{thm:generalized_closed_form_volume}. 
We also consider the relationship between the  $\bx$-parking function polytopes and partial permutahedra, allowing us to expand known results on the volume of partial permutahedra.
The relationship between partial permutahedra and $\mathbf{x}$-parking function polytopes is given by the following proposition, which answers the question of when a partial permutahedra is an $\mathbf{x}$-parking function polytope and vice versa.

\begin{proposition}\label{prop:x-pfp_pp} 
For $n=1$, $\mathcal{P}(1,p)$ is integrally equivalent to $\X_{1}(a)$ if and only if $a = p + 1$.
For $n \geq 2$,
\begin{enumerate}[(i)]
    \item if $p \geq n-1$, then $\mathcal{P}(n,p)$ is integrally equivalent to $\X_{n}(a,b)$ if and only if $b=1$ and $a = p - n +2$;
    \item if $p < n-1$, then $\mathcal{P}(n,p)$ is not integrally equivalent to any $\X_{n}(a,b)$.
\end{enumerate}

\end{proposition}




\begin{proof}
First, recall from Remark \ref{rem:n=1} that when $n=1$ for an $\mathbf{x}$-parking function, $b$ is unused. 
The result where $n=1$ is clear, as these polytopes are just 1-dimensional line segments of length $a-1 = p$.
Thus, we reduce to the case where $n \geq 2$.

For two polytopes to be integrally equivalent, they must have the same dimension (relative to its affine span). All partial permutahedra (recall that $p$ must be a positive integer) are full dimensional (Remark 5.5 \cite{HS}). As all $\X_{n}(a,b)$ are also full-dimensional, we are justified in using $n$ to denote the dimension for both.

We will show that the conditions $b=1, a=p-n+2$ are necessary for $\mathcal{P}(n,p)$ and $\Xn$ to be integrally equivalent by comparing ``maximal'' vertices. 
Apply the coordinate transformation to increase the each coordinate by one in $\mathcal{P}(n, p)$ (as in the proof of Proposition \ref{prop:classical-partial}) to match the minimal vertices of the two polytopes. 
Consider the ``maximal'' vertices \[v = (a,a+b,\ldots, a+(n-2)b, a+(n-1)b)\] of $ \Xn$ and $u= (\max(p -n +2,1), \ldots, p, p+1)$ of the shifted $\mathcal{P}(n,p)$, where ``maximal'' means the maximal vertex where all the coordinates are weakly increasing. 
If the two polytopes are integrally equivalent, we must have $v = u$.
If $b\neq 1$, the difference between the last two coordinates of $v$ is $b\geq 2$ while the difference between the last two coordinates of $u$ is 1, thus $u\neq v$. 
Hence, $b=1$ is a necessary condition.

Furthermore, consider $a \neq p -n +2$. 
If $n \leq p+1$, the first coordinate of $u$ is $\max(p -n +2,1) = p -n +2$, which gives us $u\neq v$.
If $n > p+1$, we can see that $\max(p -n +2,1) = 1$.
Additionally, the second coordinate of $u$ is at most $1$, since $(p -n +2)+1 \leq 1$.
Thus we have the first two coordinates of $u$ must be 1. 
However, the first two coordinates of $u$ are $a, a+b$, where $b$ is nonzero, thus $u\neq v$.

Hence, the conditions $b=1$, and $a = p - n +2$ are necessary for $\mathcal{P}(n,p) = \Xn$.

We next show that they are sufficient. Consider $\Xn$ where $b=1$, and $a = p - n +2$. 
Note that this second condition implies $n\leq p+1$, as $a \geq 1$.
Using the vertex description of Proposition \ref{prop:abn_vertex}, we can consider $\Xn$ to be the convex hull of all permutations of
\[ (\underbrace{1, \ldots, 1}_k, \underbrace{p -n +k+2, p -n +k+3, \ldots, p, p +1}_{n-k}),\] for all $0 \leq k \leq n$. 
Since $n\leq p+1$ implies $\min(n,p) = n$, this is exactly the vertex description of shifted $\mathcal{P}(n, p)$.

Finally, we saw that the necessary condition for integral equivalence $a = p -n+2$ implies that $n \leq p+1$. Hence if $p < n-1$, then $\mathcal{P}(n,p)$ is not integrally equivalent to any $\X_{n}(a,b)$.
\end{proof}

We now recall the recursive volume formula for the classical parking function polytope.

\begin{theorem}[Theorem 4.1, \cite{AW}]\label{thm:mit_volume}
Define a sequence $\{V_n\}_{n \geq 0}$ by $V_0=1$ and $V_n = \Vol(\PF_n)$ for all positive integers $n$.
Then $V_1 = 0$ and for all $n \geq 2$,
\begin{equation*}
    V_n = \frac{1}{n} \sum_{k=0}^{n-1} \binom{n}{k} \frac{(n-k)^{n-k-1}(n+k-1)}{2} V_k.
\end{equation*}
\end{theorem}

We are able to obtain some immediate corollaries of these results by considering dilations of $\PF_n$.

\begin{definition}
For any positive integer $d$, the \textit{$d$-dilate} of an $\mathbf{x}$-parking function polytope $\Xn$ is given by the following map on points $\varphi_d: \R^n_{\geq 1} \to \R^n_{\geq 1}$ defined by
\[ (x_1, \ldots, x_n) \mapsto (d(x_1-1) + 1, \ldots, d(x_n-1) + 1).\]
In general, a $d$-dilate of any polytope is the image of a map which (up to translation) multiplies all coordinates by $d$.
\end{definition}

\begin{lemma}\label{lem:1,b,n-dilate+a,1,n-dilate}\
\begin{enumerate}
    \item The $\bx$-parking function polytope $\X_{n}(1,b)$ is a $b$-dilate of $\X_{n}(1,1)$. 
    \item Fix $ a \geq 1$. The 
    $\bx$-parking function polytope $\X_{n}(a + (b-1)(a-1) ,b)$ is a $b$-dilate of $\X_{n}(a,1)$.
\end{enumerate}
\end{lemma}
\begin{proof}
This follows from applying the $b$-dilate map $\varphi_b$ onto all the vertices of $\X_{n}(1,1)$ and $\X_{n}(a,1)$ given by Proposition \ref{prop:abn_vertex}, which results in all the vertices of $\X_{n}(1,b)$ and $\X_{n}(a + (b-1)(a-1) ,b)$.
\end{proof}

Knowing that one polytope is a dilate of another allows for more direct approach to finding volumes for some of these parking function polytopes.

\begin{corollary}\label{cor:b^n+a,1,n-vol}\
\begin{enumerate}
    \item Fix a positive integer $b$. Then $\Vol(\X_{n}(1,b)) = b^n\Vol(\X_{n}(1,1))$.
    \item Fix positive integers $a,b$. Then $\Vol(\X_{n}(a+ (b-1)(a-1),b)) = b^n\Vol(\X_{n}(a,1))$.
\end{enumerate}
\end{corollary}

\begin{proof}
Note that the dilation of an $n$-dimensional polytope by a factor of $b$ increases the volume by a factor of $b^n$. The result then follows from Lemma \ref{lem:1,b,n-dilate+a,1,n-dilate}.
\end{proof}

Using the following result on partial permutahedra, we can then calculate the volume of more $\mathbf{x}$-parking function polytopes.
\begin{theorem}[Theorem 4.2, \cite{BCC}]
For any $n$ and $p$ with $p \geq n-1$, the normalized volume of $\mathcal{P}(n,p)$ is given recursively by
\[ \nVol(\mathcal{P}(n,p)) = (n-1)! \sum_{k=1}^n k^{k-2} \frac{\nVol(\mathcal{P}(n-k,p-k))}{(n-k)!}\left(kp-\binom{k}{2}\right) \binom{n}{k},\]
with the initial condition $\nVol(\mathcal{P}(0,p))=1$.
\end{theorem}

\begin{remark}
By Proposition \ref{prop:x-pfp_pp}, $\mathfrak{X}_{n}(a,1)$ for $a>1$ is integrally equivalent to $\mathcal{P}(n,p)$ where $p = n + a - 2 >n-1$. 
Thus, the volume $\Vol(\mathfrak{X}_{n}(a+ (b-1)(a-1),b))$ can be calculated using Corollary \ref{cor:b^n+a,1,n-vol}(2), and converting between volume and normalized volume.
\end{remark}

These results, however, do not give $\Vol(\Xn)$ for all $a,b \geq 1$. 
To do so, we need a more general construction, which is given by the following recursive volume formula.

\begin{theorem}\label{thm:generalized_recursive_volume}
Fix two positive integers $a,b$.
Define a sequence $\{V_n^{a,b}\}_{n \geq 0}$ by $V^{a,b}_0=1$ and $V^{a,b}_n = \Vol(\Xn)$ for all positive integers $n$.
Then $V^{a,b}_1 = a-1$ and for $n \geq 2$, $V^{a,b}_n$ is given recursively by 
\[ V_n^{a,b} = \frac{1}{n}\sum_{k=0}^{n-1}\binom{n}{k} \frac{(b(n-k))^{n-k-1}(nb+kb-b+2a-2)}{2}V^{a,b}_{k}.\]
\end{theorem}

In the proof of this theorem, we will use the following decomposition lemma.

\begin{lemma}[Proposition 4.1, \cite{AW}; Proposition 2, \cite{DGH}; Section 19.4, \cite{BZ}]\label{lem:decomp}
Let $K_1, \ldots, K_n$ be some convex bodies of $\R^n$ and suppose that $K_{n-m+1, \ldots, K_n}$ are contained in some $m$-dimensional affine subspace $U$ of $\R^n$. 
Let $MV_U$ denote the mixed volume with respect to the $m$-dimensional volume measure on $U$, and let $MV_{U^\perp}$ be defined similarly with respect to the orthogonal complement $U^\perp$ of $U$.
Then the mixed volume of $K_1, \ldots, K_n$
\begin{align*}
    &MV(K_1, \ldots, K_{n-m}, K_{n-m+1}, \ldots, K_n)\\
    &= \binom{n}{m}^{-1} MV_{U^\perp} (K_1', \ldots, K_{n-m}')MV_U(K_{n-m+1, \ldots, K_n),}
\end{align*}
where $K_1', \ldots, K_{n-m}'$ denote the orthogonal projections of $K_1, \ldots, K_{n-m}$ onto $U^\perp$, respectively.
\end{lemma}

We will also use the following fact.

\begin{lemma}[Lemma 4.1, \cite{BCC}]\label{lem:euclidean_volume}
The Euclidean volume of the regular permutahedron $\Pi_n \subseteq \R^n$ is $n^{n-2}\sqrt{n}$.
\end{lemma}

\begin{proof}[Proof of Theorem \ref{thm:generalized_recursive_volume}]
The case where $a=1$ follows from Theorem \ref{thm:mit_volume} and Corollary \ref{cor:b^n+a,1,n-vol}(1). 
They imply that 
\begin{align*}
    V^{1,b}_n &= b^n V_n \\
    &= \frac{b^n}{n} \sum_{k=0}^{n-1} \binom{n}{k} \frac{(n-k)^{n-k-1}(n+k-1)}{2} V_k\\
    &= \frac{b^n}{n} \sum_{k=0}^{n-1} \binom{n}{k} \frac{(n-k)^{n-k-1}(n+k-1)}{2} \frac{V^{1,b}_k}{b^k}\\
    &= \frac{1}{n} \sum_{k=0}^{n-1} \binom{n}{k} \frac{b^{n-k}(n-k)^{n-k-1}(n+k-1)}{2}V^{1,b}_k.
\end{align*}

For the case where $a>1$, we generalize the proof of Theorem 4.1 in \cite{AW}. 
Divide $\Xn$ into $n$-dimensional (full dimensional) pyramids with facets of $\Xn$ which do not contain $I = (1, \ldots, 1)$ as the base, and point $I$ as a vertex. 
For an example, see Figure \ref{fig:proof}.
Recall from Theorem \ref{thm:x-facets} that there are $2^n-1+n$ facets since $a>1$.
Of those, exactly $n$ have $I$ as a vertex, so the number of pyramids is $2^n-1$.
Each pyramid has a base which is a facet $F$ with points of $\Xn$ satisfying the equation
\[ x_{i_1} + x_{i_2} + \cdots + x_{i_k}  = ((n-k)b+a) + \cdots + ((n-2)b + a) + ((n-1)b+a),\]
for some $k \in \{1,2,\ldots, n-1, n\}$ and distinct $i_1 < \cdots < i_k$, due to the defining inequalities of $\Xn$. 

Now let $\{j_1, j_2, \ldots, j_{n-k}\} = \{1,2,\ldots, n\} \setminus \{i_1, i_2, \ldots, i_k\}$. Let $\X'_{n-k}(a,b)$ be the polytope containing all points $x'$ such that $x_p' = 0$ for all $p \in \{i_1, i_2, \ldots, i_k\}$ and for some $x \in F$, $x_p' = x_p$ for all $p \in \{j_1, j_2, \ldots, j_{n-k}\}$. 
Then $\X'_{n-k}(a,b)$ is an $(n-k)$-dimensional polytope, with the following defining inequalities:
\begin{align*}
\intertext{\emph{For all} $1 \leq p \leq n-k$,}
    1 \leq x_{j_p}' &\leq ((n-k-1)b+a),  \\
\intertext{\emph{for all} $1 \leq p < q \leq n-k$,}    
    x_{j_p}' + x_{j_q}' &\leq ((n-k-2)b+a) + ((n-k-1)b+a),\\
\intertext{\emph{for all} $1 \leq p < q < r \leq n-k$,}    
    x_{j_p}' + x_{j_q}' + x_{j_r}' &\leq ((n-k-3)b+a) + ((n-k-2)b+a) + ((n-k-1)b+a),\\
    &\vdots\\
\intertext{\emph{for all} $1 \leq p_1 < p_2 < \cdots <p_{n-k-1} \leq n-k$,}
    x_{j_{p_1}}' + x_{j_{p_2}}' + \cdots + x_{j_{p_{n-k-1}}}' &\leq (b+a) + (2b+a) + \cdots + ((n-k-1)b+a),\\
\intertext{\emph{and}}    
    x_{j_{p_1}}' + x_{j_{p_2}}' + \cdots + x_{j_{p_{n-k}}}' &\leq a + (b+a) + \cdots + ((n-k-1)b+a).
\end{align*}
By comparing with the inequality description given in Proposition \ref{prop:inequality}, we see that $\X'_{n-k}(a,b) \cong \X_{n-k}(a,b)$, as the defining inequalities are the same. 
Hence the $(n-k)$-dimensional volume $\Vol_{n-k}(\X'_{n-k}(a,b)) = \Vol_{n-k}(\X_{n-k}(a,b)) = V_{n-k}^{a,b}$.

Let $Q^{a,b}_k$ be the polytope containing all points $x'$ such that for all $p \in \{j_1, j_2, \ldots, j_{n-k}\}$, we have $x_p'=0$, and for some $x \in F$, we have $x_p' = x_p$ for all $p \in \{i_1, i_2, \ldots, i_k\}$. 
Then the coordinate values of $(x_{i_1}', x_{i_2}', \ldots, x_{i_k}')$ of the vertices of $Q^{a,b}_k$ are the permutations of $((n-k)b+a, \ldots, (n-2)b + a, (n-1)b+a)$.
Note that $Q^{a,b}_k$ is a translate of the polytope $Q^{1,b}_k$ (both are $(k-1)$-dimensional polytopes), so $Q^{a,b}_k$ and $Q^{1,b}_k$ have the same $(k-1)$-dimensional volume.
Furthermore, $Q^{1,b}_k$ is a $b$-dilate of $Q^{1,1}_k$, which is integrally equivalent to $\Pi(1, \ldots, k) = \Pi_{k-1}$, the regular permutahedron. 
As $\Pi_{k-1}$ has volume $k^{k-2} \sqrt{k}$ by Lemma \ref{lem:euclidean_volume}, then $Q^{a,b}_k$ has volume $b^{k-1}k^{k-2} \sqrt{k}$.

Thus, $F$ is a Minkowski sum of two polytopes $\X'_{n-k}(a,b)$ and $Q^{a,b}_k$ which lie in two orthogonal subspaces of $\R^n$. 
Therefore, by Lemma \ref{lem:decomp}, the volume of $F$ is
\[ \Vol(F) = \sum_{p_1, \ldots, p_n = 1}^2 MV(K_{p_1}, K_{p_2}, \ldots, K_{p_n}) = V^{a,b}_{n-k} \cdot b^{k-1}k^{k-2} \sqrt{k},\]
where $K_1 = \X_{n-k,a,b}'$ and $K_2 = Q^{a,b}_k$. 
Then the volume of $\Pyr(I,F)$, the pyramid with $I$ as the vertex over base $F$ is
\[ \Vol(\Pyr(I,F)) = \frac{1}{n}h_k\Vol(F) = \frac{1}{n}h_k V^{a,b}_{n-k} \cdot b^{k-1}k^{k-2} \sqrt{k},\]
where $h_k$ denotes the minimum distance from point $I$ to the face $F$.
We calculate that
\begin{align*}
    h_k &= \frac{|1+ \cdots + 1 -(((n-k)b+a) + \cdots + ((n-2)b + a) + ((n-1)b+a))|}{\sqrt{1+\cdots +1}}\\
    &= \frac{|k - k((2n-k-1)b+2a)/2|}{\sqrt{k}}\\
    &= \frac{k(2nb-kb-b+2a-2)}{2\sqrt{k}}.
\end{align*} 
Thus, 
\begin{align*}
    \Vol(\Pyr(I,F))
    &=\frac{1}{n}\cdot \frac{(2nb-kb-b+2a-2)}{2}V^{a,b}_{n-k} \cdot b^{k-1}k^{k-1}.
\end{align*}
By the definition of the sequence, $V^{a,b}_0=1$.
Note that $V^{a,b}_1 = a-1$, as the 1-dimensional volume (in this case length) of the convex hull of colinear points $1, \ldots, a$ in $\R$ is of length $a-1$. Thus for $n \geq 2$,
\begin{align*}
    V^{a,b}_n &= \sum_{F}\Vol(\Pyr(I,F))\\
    &= \sum_{k=1}^{n}\binom{n}{k}\frac{1}{n}\cdot \frac{(2nb-kb-b+2a-2)}{2}V^{a,b}_{n-k} b^{k-1}k^{k-1}\\
    &= \frac{1}{n}\sum_{n-k=1}^{n}\binom{n}{n-k} \frac{(2nb-(n-k)b-b+2a-2)}{2}V^{a,b}_{k} b^{n-k-1}(n-k)^{n-k-1}\\
    &= \frac{1}{n}\sum_{k=0}^{n-1}\binom{n}{k} \frac{(b(n-k))^{n-k-1}(nb+kb-b+2a-2)}{2}V^{a,b}_{k}.
\end{align*}
\end{proof}

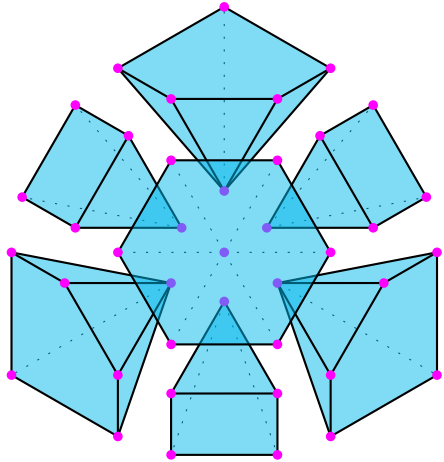
\begin{figure}[ht]
    \centering{
\begin{tikzpicture}%
	[x={(-0.707031cm, -0.408259cm)},
	y={(0.707183cm, -0.408200cm)},
	z={(0.000025cm, 0.816516cm)},
	scale=1.000000,
	back/.style={loosely dotted, thin},
	edge/.style={color=black, thick},
	facet/.style={fill=andresblue,fill opacity=0.500000},
	vertex/.style={inner sep=1pt,circle,draw=andrespink,fill=andrespink,thick}]
\coordinate (1.50000, 0.50000, 0.50000) at (1.50000, 0.50000, 0.50000);
\coordinate (4.50000, 0.50000, 0.50000) at (4.50000, 0.50000, 0.50000);
\coordinate (4.50000, 0.50000, 2.50000) at (4.50000, 0.50000, 2.50000);
\coordinate (4.50000, 1.50000, 2.50000) at (4.50000, 1.50000, 2.50000);
\coordinate (4.50000, 2.50000, 0.50000) at (4.50000, 2.50000, 0.50000);
\coordinate (4.50000, 2.50000, 1.50000) at (4.50000, 2.50000, 1.50000);
\draw[edge,back] (1.50000, 0.50000, 0.50000) -- (4.50000, 0.50000, 0.50000);
\fill[facet] (4.50000, 2.50000, 1.50000) -- (4.50000, 1.50000, 2.50000) -- (4.50000, 0.50000, 2.50000) -- (4.50000, 0.50000, 0.50000) -- (4.50000, 2.50000, 0.50000) -- cycle {};
\fill[facet] (4.50000, 1.50000, 2.50000) -- (1.50000, 0.50000, 0.50000) -- (4.50000, 0.50000, 2.50000) -- cycle {};
\fill[facet] (4.50000, 2.50000, 1.50000) -- (1.50000, 0.50000, 0.50000) -- (4.50000, 2.50000, 0.50000) -- cycle {};
\fill[facet] (4.50000, 2.50000, 1.50000) -- (1.50000, 0.50000, 0.50000) -- (4.50000, 1.50000, 2.50000) -- cycle {};
\draw[edge] (1.50000, 0.50000, 0.50000) -- (4.50000, 0.50000, 2.50000);
\draw[edge] (1.50000, 0.50000, 0.50000) -- (4.50000, 1.50000, 2.50000);
\draw[edge] (1.50000, 0.50000, 0.50000) -- (4.50000, 2.50000, 0.50000);
\draw[edge] (1.50000, 0.50000, 0.50000) -- (4.50000, 2.50000, 1.50000);
\draw[edge] (4.50000, 0.50000, 0.50000) -- (4.50000, 0.50000, 2.50000);
\draw[edge] (4.50000, 0.50000, 0.50000) -- (4.50000, 2.50000, 0.50000);
\draw[edge] (4.50000, 0.50000, 2.50000) -- (4.50000, 1.50000, 2.50000);
\draw[edge] (4.50000, 1.50000, 2.50000) -- (4.50000, 2.50000, 1.50000);
\draw[edge] (4.50000, 2.50000, 0.50000) -- (4.50000, 2.50000, 1.50000);
\node[vertex] at (1.50000, 0.50000, 0.50000)     {};
\node[vertex] at (4.50000, 0.50000, 0.50000)     {};
\node[vertex] at (4.50000, 0.50000, 2.50000)     {};
\node[vertex] at (4.50000, 1.50000, 2.50000)     {};
\node[vertex] at (4.50000, 2.50000, 0.50000)     {};
\node[vertex] at (4.50000, 2.50000, 1.50000)     {};
\coordinate (0.50000, 1.50000, 0.50000) at (0.50000, 1.50000, 0.50000);
\coordinate (0.50000, 4.50000, 0.50000) at (0.50000, 4.50000, 0.50000);
\coordinate (0.50000, 4.50000, 2.50000) at (0.50000, 4.50000, 2.50000);
\coordinate (1.50000, 4.50000, 2.50000) at (1.50000, 4.50000, 2.50000);
\coordinate (2.50000, 4.50000, 0.50000) at (2.50000, 4.50000, 0.50000);
\coordinate (2.50000, 4.50000, 1.50000) at (2.50000, 4.50000, 1.50000);
\draw[edge,back] (0.50000, 1.50000, 0.50000) -- (0.50000, 4.50000, 0.50000);
\fill[facet] (2.50000, 4.50000, 1.50000) -- (1.50000, 4.50000, 2.50000) -- (0.50000, 4.50000, 2.50000) -- (0.50000, 4.50000, 0.50000) -- (2.50000, 4.50000, 0.50000) -- cycle {};
\fill[facet] (1.50000, 4.50000, 2.50000) -- (0.50000, 1.50000, 0.50000) -- (0.50000, 4.50000, 2.50000) -- cycle {};
\fill[facet] (2.50000, 4.50000, 1.50000) -- (0.50000, 1.50000, 0.50000) -- (2.50000, 4.50000, 0.50000) -- cycle {};
\fill[facet] (2.50000, 4.50000, 1.50000) -- (0.50000, 1.50000, 0.50000) -- (1.50000, 4.50000, 2.50000) -- cycle {};
\draw[edge] (0.50000, 1.50000, 0.50000) -- (0.50000, 4.50000, 2.50000);
\draw[edge] (0.50000, 1.50000, 0.50000) -- (1.50000, 4.50000, 2.50000);
\draw[edge] (0.50000, 1.50000, 0.50000) -- (2.50000, 4.50000, 0.50000);
\draw[edge] (0.50000, 1.50000, 0.50000) -- (2.50000, 4.50000, 1.50000);
\draw[edge] (0.50000, 4.50000, 0.50000) -- (0.50000, 4.50000, 2.50000);
\draw[edge] (0.50000, 4.50000, 0.50000) -- (2.50000, 4.50000, 0.50000);
\draw[edge] (0.50000, 4.50000, 2.50000) -- (1.50000, 4.50000, 2.50000);
\draw[edge] (1.50000, 4.50000, 2.50000) -- (2.50000, 4.50000, 1.50000);
\draw[edge] (2.50000, 4.50000, 0.50000) -- (2.50000, 4.50000, 1.50000);
\node[vertex] at (0.50000, 1.50000, 0.50000)     {};
\node[vertex] at (0.50000, 4.50000, 0.50000)     {};
\node[vertex] at (0.50000, 4.50000, 2.50000)     {};
\node[vertex] at (1.50000, 4.50000, 2.50000)     {};
\node[vertex] at (2.50000, 4.50000, 0.50000)     {};
\node[vertex] at (2.50000, 4.50000, 1.50000)     {};
\coordinate (0.50000, 0.50000, 1.50000) at (0.50000, 0.50000, 1.50000);
\coordinate (0.50000, 0.50000, 4.50000) at (0.50000, 0.50000, 4.50000);
\coordinate (0.50000, 2.50000, 4.50000) at (0.50000, 2.50000, 4.50000);
\coordinate (1.50000, 2.50000, 4.50000) at (1.50000, 2.50000, 4.50000);
\coordinate (2.50000, 0.50000, 4.50000) at (2.50000, 0.50000, 4.50000);
\coordinate (2.50000, 1.50000, 4.50000) at (2.50000, 1.50000, 4.50000);
\draw[edge,back] (0.50000, 0.50000, 1.50000) -- (0.50000, 0.50000, 4.50000);
\fill[facet] (2.50000, 1.50000, 4.50000) -- (1.50000, 2.50000, 4.50000) -- (0.50000, 2.50000, 4.50000) -- (0.50000, 0.50000, 4.50000) -- (2.50000, 0.50000, 4.50000) -- cycle {};
\fill[facet] (1.50000, 2.50000, 4.50000) -- (0.50000, 0.50000, 1.50000) -- (0.50000, 2.50000, 4.50000) -- cycle {};
\fill[facet] (2.50000, 1.50000, 4.50000) -- (0.50000, 0.50000, 1.50000) -- (2.50000, 0.50000, 4.50000) -- cycle {};
\fill[facet] (2.50000, 1.50000, 4.50000) -- (0.50000, 0.50000, 1.50000) -- (1.50000, 2.50000, 4.50000) -- cycle {};
\draw[edge] (0.50000, 0.50000, 1.50000) -- (0.50000, 2.50000, 4.50000);
\draw[edge] (0.50000, 0.50000, 1.50000) -- (1.50000, 2.50000, 4.50000);
\draw[edge] (0.50000, 0.50000, 1.50000) -- (2.50000, 0.50000, 4.50000);
\draw[edge] (0.50000, 0.50000, 1.50000) -- (2.50000, 1.50000, 4.50000);
\draw[edge] (0.50000, 0.50000, 4.50000) -- (0.50000, 2.50000, 4.50000);
\draw[edge] (0.50000, 0.50000, 4.50000) -- (2.50000, 0.50000, 4.50000);
\draw[edge] (0.50000, 2.50000, 4.50000) -- (1.50000, 2.50000, 4.50000);
\draw[edge] (1.50000, 2.50000, 4.50000) -- (2.50000, 1.50000, 4.50000);
\draw[edge] (2.50000, 0.50000, 4.50000) -- (2.50000, 1.50000, 4.50000);
\node[vertex] at (0.50000, 0.50000, 1.50000)     {};
\node[vertex] at (0.50000, 0.50000, 4.50000)     {};
\node[vertex] at (0.50000, 2.50000, 4.50000)     {};
\node[vertex] at (1.50000, 2.50000, 4.50000)     {};
\node[vertex] at (2.50000, 0.50000, 4.50000)     {};
\node[vertex] at (2.50000, 1.50000, 4.50000)     {};

\coordinate (1.20000, 0.40000, 1.20000) at (1.20000, 0.40000, 1.20000);
\coordinate (3.20000, 0.40000, 4.20000) at (3.20000, 0.40000, 4.20000);
\coordinate (3.20000, 1.40000, 4.20000) at (3.20000, 1.40000, 4.20000);
\coordinate (4.20000, 0.40000, 3.20000) at (4.20000, 0.40000, 3.20000);
\coordinate (4.20000, 1.40000, 3.20000) at (4.20000, 1.40000, 3.20000);
\draw[edge,back] (1.20000, 0.40000, 1.20000) -- (3.20000, 0.40000, 4.20000);
\draw[edge,back] (1.20000, 0.40000, 1.20000) -- (4.20000, 0.40000, 3.20000);
\fill[facet] (4.20000, 1.40000, 3.20000) -- (3.20000, 1.40000, 4.20000) -- (3.20000, 0.40000, 4.20000) -- (4.20000, 0.40000, 3.20000) -- cycle {};
\fill[facet] (4.20000, 1.40000, 3.20000) -- (1.20000, 0.40000, 1.20000) -- (3.20000, 1.40000, 4.20000) -- cycle {};
\draw[edge] (1.20000, 0.40000, 1.20000) -- (3.20000, 1.40000, 4.20000);
\draw[edge] (1.20000, 0.40000, 1.20000) -- (4.20000, 1.40000, 3.20000);
\draw[edge] (3.20000, 0.40000, 4.20000) -- (3.20000, 1.40000, 4.20000);
\draw[edge] (3.20000, 0.40000, 4.20000) -- (4.20000, 0.40000, 3.20000);
\draw[edge] (3.20000, 1.40000, 4.20000) -- (4.20000, 1.40000, 3.20000);
\draw[edge] (4.20000, 0.40000, 3.20000) -- (4.20000, 1.40000, 3.20000);
\node[vertex] at (1.20000, 0.40000, 1.20000)     {};
\node[vertex] at (3.20000, 0.40000, 4.20000)     {};
\node[vertex] at (3.20000, 1.40000, 4.20000)     {};
\node[vertex] at (4.20000, 0.40000, 3.20000)     {};
\node[vertex] at (4.20000, 1.40000, 3.20000)     {};

\coordinate (1.20000, 1.20000, 0.40000) at (1.20000, 1.20000, 0.40000);
\coordinate (3.20000, 4.20000, 0.40000) at (3.20000, 4.20000, 0.40000);
\coordinate (3.20000, 4.20000, 1.40000) at (3.20000, 4.20000, 1.40000);
\coordinate (4.20000, 3.20000, 0.40000) at (4.20000, 3.20000, 0.40000);
\coordinate (4.20000, 3.20000, 1.40000) at (4.20000, 3.20000, 1.40000);
\draw[edge,back] (1.20000, 1.20000, 0.40000) -- (3.20000, 4.20000, 0.40000);
\draw[edge,back] (1.20000, 1.20000, 0.40000) -- (4.20000, 3.20000, 0.40000);
\fill[facet] (4.20000, 3.20000, 1.40000) -- (3.20000, 4.20000, 1.40000) -- (3.20000, 4.20000, 0.40000) -- (4.20000, 3.20000, 0.40000) -- cycle {};
\fill[facet] (4.20000, 3.20000, 1.40000) -- (1.20000, 1.20000, 0.40000) -- (3.20000, 4.20000, 1.40000) -- cycle {};
\draw[edge] (1.20000, 1.20000, 0.40000) -- (3.20000, 4.20000, 1.40000);
\draw[edge] (1.20000, 1.20000, 0.40000) -- (4.20000, 3.20000, 1.40000);
\draw[edge] (3.20000, 4.20000, 0.40000) -- (3.20000, 4.20000, 1.40000);
\draw[edge] (3.20000, 4.20000, 0.40000) -- (4.20000, 3.20000, 0.40000);
\draw[edge] (3.20000, 4.20000, 1.40000) -- (4.20000, 3.20000, 1.40000);
\draw[edge] (4.20000, 3.20000, 0.40000) -- (4.20000, 3.20000, 1.40000);
\node[vertex] at (1.20000, 1.20000, 0.40000)     {};
\node[vertex] at (3.20000, 4.20000, 0.40000)     {};
\node[vertex] at (3.20000, 4.20000, 1.40000)     {};
\node[vertex] at (4.20000, 3.20000, 0.40000)     {};
\node[vertex] at (4.20000, 3.20000, 1.40000)     {};

\coordinate (0.40000, 1.20000, 1.20000) at (0.40000, 1.20000, 1.20000);
\coordinate (0.40000, 3.20000, 4.20000) at (0.40000, 3.20000, 4.20000);
\coordinate (0.40000, 4.20000, 3.20000) at (0.40000, 4.20000, 3.20000);
\coordinate (1.40000, 3.20000, 4.20000) at (1.40000, 3.20000, 4.20000);
\coordinate (1.40000, 4.20000, 3.20000) at (1.40000, 4.20000, 3.20000);
\draw[edge,back] (0.40000, 1.20000, 1.20000) -- (0.40000, 3.20000, 4.20000);
\draw[edge,back] (0.40000, 1.20000, 1.20000) -- (0.40000, 4.20000, 3.20000);

\fill[facet] (1.40000, 4.20000, 3.20000) -- (0.40000, 4.20000, 3.20000) -- (0.40000, 3.20000, 4.20000) -- (1.40000, 3.20000, 4.20000) -- cycle {};
\fill[facet] (1.40000, 4.20000, 3.20000) -- (0.40000, 1.20000, 1.20000) -- (1.40000, 3.20000, 4.20000) -- cycle {};
\draw[edge] (0.40000, 1.20000, 1.20000) -- (1.40000, 3.20000, 4.20000);
\draw[edge] (0.40000, 1.20000, 1.20000) -- (1.40000, 4.20000, 3.20000);
\draw[edge] (0.40000, 3.20000, 4.20000) -- (0.40000, 4.20000, 3.20000);
\draw[edge] (0.40000, 3.20000, 4.20000) -- (1.40000, 3.20000, 4.20000);
\draw[edge] (0.40000, 4.20000, 3.20000) -- (1.40000, 4.20000, 3.20000);
\draw[edge] (1.40000, 3.20000, 4.20000) -- (1.40000, 4.20000, 3.20000);
\node[vertex] at (0.40000, 1.20000, 1.20000)     {};
\node[vertex] at (0.40000, 3.20000, 4.20000)     {};
\node[vertex] at (0.40000, 4.20000, 3.20000)     {};
\node[vertex] at (1.40000, 3.20000, 4.20000)     {};
\node[vertex] at (1.40000, 4.20000, 3.20000)     {};

\coordinate (1.00000, 1.00000, 1.00000) at (1.00000, 1.00000, 1.00000);
\coordinate (2.00000, 3.00000, 4.00000) at (2.00000, 3.00000, 4.00000);
\coordinate (2.00000, 4.00000, 3.00000) at (2.00000, 4.00000, 3.00000);
\coordinate (3.00000, 2.00000, 4.00000) at (3.00000, 2.00000, 4.00000);
\coordinate (3.00000, 4.00000, 2.00000) at (3.00000, 4.00000, 2.00000);
\coordinate (4.00000, 2.00000, 3.00000) at (4.00000, 2.00000, 3.00000);
\coordinate (4.00000, 3.00000, 2.00000) at (4.00000, 3.00000, 2.00000);
\draw[edge,back] (1.00000, 1.00000, 1.00000) -- (2.00000, 3.00000, 4.00000);
\draw[edge,back] (1.00000, 1.00000, 1.00000) -- (2.00000, 4.00000, 3.00000);
\draw[edge,back] (1.00000, 1.00000, 1.00000) -- (3.00000, 2.00000, 4.00000);
\draw[edge,back] (1.00000, 1.00000, 1.00000) -- (3.00000, 4.00000, 2.00000);
\draw[edge,back] (1.00000, 1.00000, 1.00000) -- (4.00000, 2.00000, 3.00000);
\draw[edge,back] (1.00000, 1.00000, 1.00000) -- (4.00000, 3.00000, 2.00000);
\node[vertex] at (1.00000, 1.00000, 1.00000)     {};
\fill[facet] (4.00000, 3.00000, 2.00000) -- (3.00000, 4.00000, 2.00000) -- (2.00000, 4.00000, 3.00000) -- (2.00000, 3.00000, 4.00000) -- (3.00000, 2.00000, 4.00000) -- (4.00000, 2.00000, 3.00000) -- cycle {};
\draw[edge] (2.00000, 3.00000, 4.00000) -- (2.00000, 4.00000, 3.00000);
\draw[edge] (2.00000, 3.00000, 4.00000) -- (3.00000, 2.00000, 4.00000);
\draw[edge] (2.00000, 4.00000, 3.00000) -- (3.00000, 4.00000, 2.00000);
\draw[edge] (3.00000, 2.00000, 4.00000) -- (4.00000, 2.00000, 3.00000);
\draw[edge] (3.00000, 4.00000, 2.00000) -- (4.00000, 3.00000, 2.00000);
\draw[edge] (4.00000, 2.00000, 3.00000) -- (4.00000, 3.00000, 2.00000);
\node[vertex] at (2.00000, 3.00000, 4.00000)     {};
\node[vertex] at (2.00000, 4.00000, 3.00000)     {};
\node[vertex] at (3.00000, 2.00000, 4.00000)     {};
\node[vertex] at (3.00000, 4.00000, 2.00000)     {};
\node[vertex] at (4.00000, 2.00000, 3.00000)     {};
\node[vertex] at (4.00000, 3.00000, 2.00000)     {};
\end{tikzpicture}
}
    \caption{The $\mathbf{x}$-parking function polytope $\mathfrak{X}_3(2,1)$ where, as in the proof of Theorem \ref{thm:generalized_recursive_volume}, it has been split up into pyramids with $I = (1,1,1)$ as the vertex over each face $F$ which does not contain $I$. By summing the volume of all pyramids, we obtain the volume of the whole polytope.}
    \label{fig:proof}
\end{figure}

\begin{reptheorem}{thm:generalized_closed_form_volume}
For any positive integers $a,b,n$, the normalized volume $\nVol(\mathfrak{X}_{n}(a,b))$ is given by
\begin{equation}
    \nVol(\mathfrak{X}_{n}(a,b)) = -n!\left(\frac{b}{2}\right)^n \sum_{i=0}^n \binom{n}{i} (2i-3)!! \left(2n-1 + \frac{2a-2}{b}\right)^{n-i}. \nonumber
\end{equation}
\end{reptheorem}

To prove this, we use the following lemma.

\begin{lemma}\label{lemma:exp-g_b}
The exponential generating function for $\{(bn)^{n-1}\}_{n \geq 1}$
\begin{equation}
    g_b(x) := \sum_{n \geq 1} \frac{(bn)^{n-1}}{n!} x^n 
\end{equation}
satisfies
\begin{equation}\label{eq:b-Lambert}
    g_b(x) = -\frac{1}{b}W_0(-bx), 
\end{equation}
where $W_0$ denotes the principal branch of the Lambert $W$ function, and 
\begin{equation}\label{eq:g_b(x)}
    g_b(x) = xe^{bg_b(x)}.
\end{equation}
\end{lemma}

\begin{proof}
Recall that 
\begin{equation}
    W_0(z) = \sum_{n \geq 1} \frac{(-n)^{n-1}}{n!} z^n, \nonumber
\end{equation}
so substituting $z=-bx$ gives
\begin{equation}\label{eq:W_0(-bx)}
    W_0(-bx) = \sum_{n \geq 1} \frac{(-n)^{n-1}}{n!} (-bx)^n =  -b\sum_{n \geq 1} \frac{(bn)^{n-1}}{n!} x^n,
\end{equation}
which implies Equation (\ref{eq:b-Lambert}). A well-known property of the Lambert $W$ function is that 
$W_0(z) = ze^{-W_0(z)}$.
Substituting $z=-bx$ gives    
$W_0(-bx) = -bxe^{-W_0(-bx)}$,
which by Equation (\ref{eq:b-Lambert}) gives
$-bg_b(x) = -bxe^{bg_b(x)}$,
which implies Equation (\ref{eq:g_b(x)}).
\end{proof}

We also need the following exponential generating function for the (non-normalized) volume of $\mathfrak{X}_{n,a,b}$. This is a generalization of Proposition 4.2 in \cite{AW}.

\begin{proposition}\label{prop:general_exp_gen}
Let $V^{a,b}_n$ denote the Euclidean volume of $\mathfrak{X}_{n}(a,b)$. Let
\begin{equation}
    f_{a,b}(x) := \sum_{n\geq 0} \frac{V^{a,b}_n}{n!}x^n
\end{equation}
be its exponential generating function. Then 
\begin{equation}
    f_{a,b}(x) = e^{\frac{b^2}{2}\int x(g_b'(x))^2 \, dx}e^{(a-1)g_b(x)},
\end{equation}
where $g_b(x)$ is the exponential generating function given in Lemma \ref{lemma:exp-g_b}.
\end{proposition}

\begin{proof}
By Theorem \ref{thm:generalized_recursive_volume}, we have that
\begin{align*}
    n \cdot \frac{V^{a,b}_n}{n!} &= \frac{1}{n!} \sum_{k=0}^{n-1} \binom{n}{k}\frac{(b(n-k))^{n-k-1}(nb+kb-b+2a-2)}{2} V^{a,b}_k\\
    &= \sum_{k=0}^{n-1} \frac{(b(n-k))^{n-k-1}(b(n-k) + 2kb + (2a-b-2))}{2(n-k)!}\frac{V^{a,b}_k}{k!}\\
    &= \sum_{k=0}^{n-1}\frac{b}{2} \frac{b^{n-k-1}(n-k)^{n-k}}{(n-k)!}\frac{V^{a,b}_k}{k!} +  \sum_{k=0}^{n-1}\frac{(b(n-k))^{n-k-1}bk}{(n-k)!}\frac{V^{a,b}_k}{k!}\\
    &+ \sum_{k=0}^{n-1}\frac{2a-b-2}{2} \frac{(b(n-k))^{n-k-1}}{(n-k)!}\frac{V^{a,b}_k}{k!},
\end{align*}
which by summing over all $x^n$ gives
\begin{equation}\label{eq:first}
    f_{a,b}'(x) = \frac{b}{2}g_b'(x)f_{a,b}(x) + bg_b(x)f_{a,b}'(x) + \frac{2a-b-2}{2x}g_b(x)f_{a,b}(x).
\end{equation}
Note that by Equation (\ref{eq:g_b(x)}),
\begin{align*}
    g_b'(x) &= (xe^{bg_b(x)})'\\
    &= e^{bg_b(x)} + bxg_b'(x)e^{bg_b(x)}\\
    &= \frac{g_b(x)}{x} + bg_b'(x)g_b(x),
\end{align*}
which implies that
\begin{equation}\label{eq:star_b}
    xg_b'(x) - g_b(x) = bxg_b'(x)g_b(x),
\end{equation}
and
\begin{equation}\label{eq:star_2}
    1 - bg_b(x) = \frac{g_b(x)}{xg_b'(x)}.
\end{equation}
Thus by Equation (\ref{eq:first}),
\begin{align*}
    f_{a,b}'(x)(1- bg_b(x)) &= \frac{b}{2}g_b'(x)f_{a,b}(x) + \frac{2a-b-2}{2x}g_b(x)f_{a,b}(x)\\
    &= \frac{1}{2x}(bxg_b'(x) + (2a-b-2)g_b(x))f_{a,b}(x)\\
    &= \frac{1}{2x}(b^2xg_b'(x)g_b(x) + (2a-2)g_b(x))f_{a,b}(x),
\end{align*}
where the last line follows by Equation (\ref{eq:star_b}). Then by dividing both sides by Equation (\ref{eq:star_2}), we get
\begin{align*}
    f_{a,b}'(x) &= \frac{xg_b'(x)}{g_b(x)}\frac{1}{2x}(b^2xg_b'(x)g_b(x) + (2a-2)g_b(x))f_{a,b}(x)\\
    &= \left( \frac{b^2}{2} x(g_b'(x))^2 + (a-1)g_b'(x) \right) f_{a,b}(x).
\end{align*}
This differential equation has solution
\begin{align*}
    f_{a,b}(x) &= e^{\int \frac{b^2}{2}x (g_b'(x))^2 + (a-1)g_b'(x)\, dx}\\
    &= e^{\frac{b^2}{2}\int x(g_b'(x))^2 \, dx}e^{(a-1)g_b(x)}.
\end{align*}
\end{proof}

Our proof will also use the following theorem (see for example Chapter 4, Entry 11, \cite{Ram}).

\begin{theorem}[Ramanujan's Master Theorem]\label{thm:RMT} Let $\Gamma(s)$ denote the gamma function. If \begin{equation}\label{eq:RMT_f} f(x) = \sum_{n=0}^\infty \frac{\varphi(n)}{n!}(-x)^n,\end{equation}
then the Mellin transform of $f(x)$ is given by 
\begin{equation}\label{eq:RMT_mellin} \int_0^\infty x^{s-1} f(x) \, dx = \Gamma(s)\varphi(-s).\end{equation}
\end{theorem}

Now we are ready to prove our theorem.
\begin{proof}[Proof of Theorem \ref{thm:generalized_closed_form_volume}]
By Ramanujan's Master Theorem, for $\varphi(n) := (-1)^nV^{a,b}_n$,
\begin{equation}
    \frac{1}{\Gamma(s)}\int_0^\infty x^{s-1} f_{a,b}(x) \, dx = \varphi(-s) =(-1)^{-s} V^{a,b}_{-s}, \nonumber
\end{equation}
where $\Gamma(s)$ denotes the gamma function. By taking the limit $s \to -n$, we get 
\begin{equation}
    \lim_{s \to -n} \frac{1}{\Gamma(s)} \int_0^\infty x^{s-1} f_{a,b}(x) \, dx= (-1)^{n} V^{a,b}_{n}. \nonumber
\end{equation}
By Equation (\ref{eq:g_b(x)}), we have that
\begin{equation}\label{eq:newstar_1}
x=g_b(x)e^{-bg_b(x)} \qquad \text{and} \qquad dx = (1-bg_b )e^{-bg_b}\, dg_b.
\end{equation}
By Equations (\ref{eq:star_2}) and (\ref{eq:g_b(x)}), we have that 
\begin{equation}
    (g_b'(x))^2 = \left( \frac{g_b(x)}{x(1-bg_b(x))} \right)^2 = \left(\frac{xe^{bg_b(x)}}{x(1-bg_b(x))}\right)^2 = \left(\frac{e^{bg_b(x)}}{(1-bg_b(x))}\right)^2. \nonumber
\end{equation}
Thus
\begin{align*}
    x (g_b'(x))^2 \, dx &= g_be^{-bg_b} \left(\frac{e^{bg_b}}{(1-bg_b)}\right)^2 (1-bg_b )e^{-bg_b}\, dg_b\nonumber\\
    &= \frac{g_b}{1-bg_b}\, dg_b.
\end{align*}
By Proposition \ref{prop:general_exp_gen}, we have
\begin{align}\label{eq:newstar_2}
    f_{a,b}(x)
    &= e^{\frac{b^2}{2}\int x(g_b'(x))^2 \, dx}e^{(a-1)g_b(x)}\nonumber\\
    &= e^{\frac{b^2}{2}\int \frac{g_b}{1-bg_b}\, dg_b}e^{(a-1)g_b}\nonumber\\
    &= e^{-\frac{b}{2}g_b} e^{-\frac{1}{2}\ln(1-bg_b)}e^{(a-1)g_b}\nonumber\\
    &= \frac{1}{\sqrt{1-bg_b}}e^{(a-\frac{b}{2}-1)g_b}.
\end{align}
Thus by Equations (\ref{eq:newstar_1}) and (\ref{eq:newstar_2}), we have that
\begin{align*}
    (-1)^n V^{a,b}_n &= \lim_{s \to -n} \frac{1}{\Gamma(s)} \int_0^{-\infty} (g_b e^{-bg_b})^{s-1} \frac{e^{(a-\frac{b}{2}-1)g_b}}{\sqrt{1-bg_b}} (1-bg_b )e^{-bg_b}\, dg_b\\
    &= \lim_{s \to -n} \frac{1}{\Gamma(s)} \int_0^{-\infty} g_b^{s-1} \sqrt{1-bg_b} e^{(-sb+a-\frac{b}{2}-1)g_b}\, dg_b
    \intertext{and by replacing $g_b$ by $-t$ and $dg_b$ by $-dt$,}
    &= \lim_{s \to -n} \frac{1}{\Gamma(s)} \int_0^{\infty} (-1)^{s}t^{s-1} \sqrt{1+bt} e^{-(-sb+a-\frac{b}{2}-1)t}\, dt,
\end{align*}
which implies that
\begin{equation}
    V^{a,b}_n = \lim_{s \to -n} \frac{1}{\Gamma(s)} \int_0^\infty t^{s-1}\sqrt{1+bt} e^{-(-sb+a-\frac{b}{2}-1)t}\, dt. \nonumber
\end{equation}
Note that as a formal power series,
\begin{equation}
    \sqrt{1+bt} = \sum_{\ell=0}^\infty \frac{(-1)^\ell (bt)^\ell (-\frac{1}{2})_\ell}{\ell!}, \nonumber
\end{equation}
where $(-\frac{1}{2})_\ell$ is the Pochhammer symbol, defined by $(\lambda)_\ell := \lambda (\lambda + 1) \cdots (\lambda + \ell -1)$, for any positive integer $\ell$, and by $(\lambda)_0:=1$. So,
\begin{align*}
    V^{a,b}_n &= \lim_{s \to -n} \frac{1}{\Gamma(s)}\int_0^\infty t^{s-1} e^{-(-sb+a-\frac{b}{2}-1)t}\sum_{\ell=0}^\infty \frac{(-1)^\ell (bt)^\ell (-\frac{1}{2})_\ell}{\ell!}\, dt\\
    &= \sum_{\ell=0}^\infty \frac{(-1)^\ell b^\ell (-\frac{1}{2})_\ell}{\ell!} \lim_{s \to -n} \frac{1}{\Gamma(s)}\int_0^\infty t^{\ell+s-1} e^{-(-sb+a-\frac{b}{2}-1)t}\, dt.
\end{align*}
Let $u = (-sb+a-\frac{b}{2} - 1)t$ so $du = (-sb+a-\frac{b}{2} - 1)dt$. Then
\begin{align*}
    \lim_{s \to -n} \frac{1}{\Gamma(s)} \int_0^\infty t^{\ell+s-1} e^{-(-sb+a-\frac{b}{2}-1)t}\, dt 
    &= \lim_{s \to -n} \frac{1}{\Gamma(s)} \int_0^\infty \left(\frac{u}{-sb+a-\frac{b}{2}-1}\right)^{\ell+s-1} \frac{e^{-u}}{-sb+a-\frac{b}{2} - 1}\, du\\
    &= \lim_{s \to -n} \frac{1}{\Gamma(s)}\int_0^\infty e^{-u} \frac{u^{\ell+s-1}}{(-sb+a-\frac{b}{2}-1)^{\ell+s}}\\
    &= \lim_{s \to -n} \frac{\Gamma(\ell+s)}{\Gamma(s)}\frac{1}{(-sb+a-\frac{b}{2}-1)^{\ell+s}}\\
    &= \lim_{s \to -n} \frac{(s)_{\ell}}{(-sb+a-\frac{b}{2}-1)^{\ell+s}}\\
    &= \frac{(-n)_{\ell}}{(nb+a-\frac{b}{2}-1)^{\ell-n}}.
\end{align*}
Thus,
\begin{align*}
    V^{a,b}_n
    &= \sum_{\ell=0}^\infty \frac{(-1)^\ell b^\ell (-\frac{1}{2})_\ell (-n)_\ell}{\ell! (nb+a-\frac{b}{2}-1)^{\ell-n}}\\
    &= (nb+a - \frac{b}{2} -1)^n \sum_{\ell=0}^\infty \frac{(-\frac{1}{2})_\ell (-n)_\ell}{\ell!}\left( \frac{-1}{n+\frac{a}{b}-\frac{1}{2}-\frac{1}{b}}\right)^\ell,
\end{align*}
which is a multiple of a Poisson-Charlier polynomial. See for example \cite{OE} for the following facts about the Poisson-Charlier polynomial. 
\begin{align*}
    C_n\left(\frac{1}{2}; n + \frac{a}{b} - \frac{1}{2} - \frac{1}{b}\right) 
    &= \sum_{i=0}^n(-1)^i \binom{n}{i}\binom{\frac{1}{2}}{i}i!\left(n + \frac{a}{b} - \frac{1}{2} - \frac{1}{b}\right)^{-i}\\
    &= {}_2F_0\left(-n; -\frac{1}2{}; - ; \frac{-1}{n + \frac{a}{b} - \frac{1}{2} - \frac{1}{b}}\right)\\
    &= \sum_{\ell=0}^\infty \frac{ (-\frac{1}{2})_\ell(-n)_\ell}{\ell!} \left( \frac{-1}{n + \frac{a}{b} - \frac{1}{2} - \frac{1}{b}}\right)^\ell.
\end{align*}
Thus,
\begingroup
\allowdisplaybreaks
\begin{align*}
    V^{a,b}_n &= (nb +a - \frac{b}{2}-1)^n C_n \left(\frac{1}{2}; n + \frac{a}{b} - \frac{1}{2} - \frac{1}{b}\right)\\
    &= b^n\left(n +\frac{a}{b} - \frac{1}{2}-\frac{1}{b}\right)^n  \sum_{i=0}^n(-1)^i \binom{n}{i}\binom{\frac{1}{2}}{i}i!\left(n + \frac{a}{b} - \frac{1}{2} - \frac{1}{b}\right)^{-i}\\
    &= b^n \sum_{i=0}^n(-1)^i \binom{n}{i}\binom{2i}{i} \frac{(-1)^{i+1}}{2^{2i}(2i-1)}i!\left(n + \frac{a}{b} - \frac{1}{2} - \frac{1}{b}\right)^{n-i}\\
    &= -b^n \sum_{i=0}^n \binom{n}{i}\frac{(2i)!}{i!2^{i}} \frac{1}{2^i(2i-1)}\left(n + \frac{a}{b} - \frac{1}{2} - \frac{1}{b}\right)^{n-i}\\
    &= -b^n \sum_{i=0}^n \binom{n}{i}(2i-1)!!\frac{1}{2^i(2i-1)}2^{i-n}\left(2n - 1 + \frac{2a-2}{b}\right)^{n-i}\\
    &= -\left(\frac{b}{2}\right)^n \sum_{i=0}^n \binom{n}{i}(2i-3)!!\left(2n - 1 + \frac{2a-2}{b}\right)^{n-i},
\end{align*}
\endgroup
and to get the normalized volume, we simply multiply by $n!$, which gives the desired formula.
\end{proof}

As a corollary, we prove the following conjecture of Behrend et al. \cite{BCC}.\footnote{This appeared as a conjecture in their first arXiv preprint. Since the appearance of our paper, they have provided an alternative proof.}

\begin{corollary}[Conjecture 4.5 and Remark 4.6, \cite{BCC}]\label{cor:conjecture}
For any $n$ and $p$, with $p \geq n-1$, 
\begin{equation}
    \nVol(\mathcal{P}(n,p)) = n! \sum_{k=0}^n \binom{n}{k}\frac{c_k}{2^k}p^{n-k},
\end{equation}
where the sequence $c_k$ satisfies $c_k = 2(k-1)(c_{k-1}-c_{k-2})$ and has exponential generating function $\sqrt{1-2x}e^x$. Equivalently,
\begin{equation}\label{eq:red_conjecture}
    \nVol(\mathcal{P}(n,p)) = - \frac{n!}{2^n}\sum_{i=0}^n \binom{n}{i}(2i-3)!!(2p+1)^{n-i}.
\end{equation}
\end{corollary}

\begin{proof}
The two statements are shown to be equivalent in \cite{BCC}, so we prove the latter.
By Proposition \ref{prop:x-pfp_pp}, for $p \geq n-1$, $\nVol(\mathcal{P}(n,p) = \nVol(\mathfrak{X}_{n}(p-n+2,1))$. 
Then by Theorem \ref{thm:generalized_closed_form_volume},
\begin{align*}
    \nVol(\mathfrak{X}_{n,p-n+2,1}) &= -n!\left(\frac{1}{2}\right)^n \sum_{i=0}^n \binom{n}{i} (2i-3)!! \left(2n-1 + 2(p-n+2)-2\right)^{n-i}\\
    &= -\frac{n!}{2^n} \sum_{i=0}^n \binom{n}{i} (2i-3)!! (2p+1)^{n-i},
\end{align*}
as desired.
\end{proof}



\section{The convex hull of weakly increasing \texorpdfstring{$\mathbf{x}$}{\textbf{x}}-parking functions}\label{sec:weakly_increasing}

\begin{definition}
A weakly increasing $\mathbf{x}$-parking function associated to a positive integer vector $\mathbf{x}$  is a weakly increasing sequence $(a_1, a_2, \ldots, a_n)$ of positive integers which satisfy $a_i \leq x_1+\cdots+x_i$.
\end{definition}

We will only focus on weakly increasing $\mathbf{x}$-parking functions associated to vectors $\mathbf{x} = (a,b,b, \ldots, b)$. Denote the weakly increasing $\mathbf{x}$-parking function polytope associated to a positive integer vector of the form $(a,b,b,\ldots, b)$ by $\X_{n}^w(a,b)$, where $n$ is the length of the vector. Note that the weakly increasing $\mathbf{x}$-parking functions are just the subset of the $\mathbf{x}$-parking functions which are weakly increasing.
Note that $\X^w_n (1,b)$ is an $(n-1)$-dimensional polytope in $\R^n$ since the first coordinate of any weakly increasing parking function is $1$, that is $x_1=1$, which brings down the dimension of the polytope by 1. See Figure \ref{fig:weakly} for examples.

Next, we reveal a connection between the weakly increasing $\bx$-parking function polytope and the Pitman-Stanley polytope.
The Pitman-Stanley polytope is a well-studied polytope, which has connections to flow polytopes, parking functions, and many other combinatorial objects. 

\begin{definition}
    For any $\bx\in \R^n$, the Pitman-Stanley polytope $\mathsf{PS}_n(\bx) $ is defined to be \[\{\by \in \R^n : y_i \geq 0 \text{ and } y_1+\cdots + y_i \leq x_1 + \cdots + x_i \text{ for all } 1\leq i \leq n\}.\]
\end{definition}

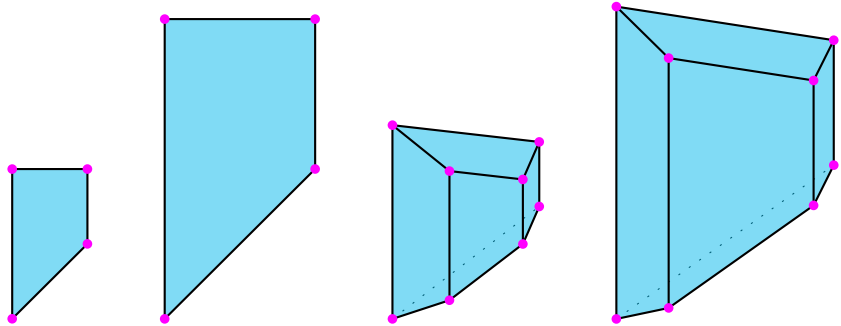
\begin{figure}[h]
    \centering
\begin{tikzpicture}
	[x={(0.003058cm, -0.089913cm)},
	y={(0.999995cm, 0.000293cm)},
	z={(-0.000018cm, 0.995950cm)},
	scale=1.000000,
	back/.style={loosely dotted, thin},
	edge/.style={color=black, thick},
	facet/.style={fill=andresblue,fill opacity=0.500000},
	vertex/.style={inner sep=1pt,circle,draw=andrespink,fill=andrespink,thick}]
%
%

\coordinate (1.00000, 1.00000, 1.00000) at (1.00000, 1.00000, 1.00000);
\coordinate (1.00000, 2.00000, 3.00000) at (1.00000, 2.00000, 3.00000);
\coordinate (1.00000, 1.00000, 3.00000) at (1.00000, 1.00000, 3.00000);
\coordinate (1.00000, 2.00000, 2.00000) at (1.00000, 2.00000, 2.00000);
\fill[facet] (1.00000, 2.00000, 3.00000) -- (1.00000, 2.00000, 2.00000) -- (1.00000, 1.00000, 1.00000) -- (1.00000, 1.00000, 3.00000) -- cycle {};
\draw[edge] (1.00000, 1.00000, 1.00000) -- (1.00000, 1.00000, 3.00000);
\draw[edge] (1.00000, 1.00000, 1.00000) -- (1.00000, 2.00000, 2.00000);
\draw[edge] (1.00000, 2.00000, 3.00000) -- (1.00000, 1.00000, 3.00000);
\draw[edge] (1.00000, 2.00000, 3.00000) -- (1.00000, 2.00000, 2.00000);
\node[vertex] at (1.00000, 1.00000, 1.00000)     {};
\node[vertex] at (1.00000, 2.00000, 3.00000)     {};
\node[vertex] at (1.00000, 1.00000, 3.00000)     {};
\node[vertex] at (1.00000, 2.00000, 2.00000)     {};
\end{tikzpicture}
\qquad
\begin{tikzpicture}
	[x={(-0.001611cm, -0.075149cm)},
	y={(0.999999cm, -0.000130cm)},
	z={(0.000009cm, 0.997172cm)},
	scale=1.000000,
	back/.style={loosely dotted, thin},
	edge/.style={color=black, thick},
	facet/.style={fill=andresblue,fill opacity=0.500000},
	vertex/.style={inner sep=1pt,circle,draw=andrespink,fill=andrespink,thick}]
%
%

\coordinate (1.00000, 1.00000, 1.00000) at (1.00000, 1.00000, 1.00000);
\coordinate (1.00000, 3.00000, 5.00000) at (1.00000, 3.00000, 5.00000);
\coordinate (1.00000, 3.00000, 3.00000) at (1.00000, 3.00000, 3.00000);
\coordinate (1.00000, 1.00000, 5.00000) at (1.00000, 1.00000, 5.00000);
\fill[facet] (1.00000, 3.00000, 5.00000) -- (1.00000, 1.00000, 5.00000) -- (1.00000, 1.00000, 1.00000) -- (1.00000, 3.00000, 3.00000) -- cycle {};
\draw[edge] (1.00000, 1.00000, 1.00000) -- (1.00000, 3.00000, 3.00000);
\draw[edge] (1.00000, 1.00000, 1.00000) -- (1.00000, 1.00000, 5.00000);
\draw[edge] (1.00000, 3.00000, 5.00000) -- (1.00000, 3.00000, 3.00000);
\draw[edge] (1.00000, 3.00000, 5.00000) -- (1.00000, 1.00000, 5.00000);
\node[vertex] at (1.00000, 1.00000, 1.00000)     {};
\node[vertex] at (1.00000, 3.00000, 5.00000)     {};
\node[vertex] at (1.00000, 3.00000, 3.00000)     {};
\node[vertex] at (1.00000, 1.00000, 5.00000)     {};
\end{tikzpicture}
\qquad
\begin{tikzpicture}
	[x={(-0.217826cm, -0.499310cm)},
	y={(0.975988cm, -0.111370cm)},
	z={(-0.000078cm, 0.859236cm)},
	scale=1.000000,
	back/.style={loosely dotted, thin},
	edge/.style={color=black, thick},
	facet/.style={fill=andresblue,fill opacity=0.500000},
	vertex/.style={inner sep=1pt,circle,draw=andrespink,fill=andrespink,thick}]
%

\coordinate (1.00000, 1.00000, 1.00000) at (1.00000, 1.00000, 1.00000);
\coordinate (2.00000, 3.00000, 4.00000) at (2.00000, 3.00000, 4.00000);
\coordinate (2.00000, 3.00000, 3.00000) at (2.00000, 3.00000, 3.00000);
\coordinate (1.00000, 1.00000, 4.00000) at (1.00000, 1.00000, 4.00000);
\coordinate (2.00000, 2.00000, 4.00000) at (2.00000, 2.00000, 4.00000);
\coordinate (2.00000, 2.00000, 2.00000) at (2.00000, 2.00000, 2.00000);
\coordinate (1.00000, 3.00000, 4.00000) at (1.00000, 3.00000, 4.00000);
\coordinate (1.00000, 3.00000, 3.00000) at (1.00000, 3.00000, 3.00000);
\draw[edge,back] (1.00000, 1.00000, 1.00000) -- (1.00000, 3.00000, 3.00000);
\fill[facet] (1.00000, 1.00000, 4.00000) -- (1.00000, 3.00000, 4.00000) -- (2.00000, 3.00000, 4.00000) -- (2.00000, 2.00000, 4.00000) -- cycle {};
\fill[facet] (1.00000, 3.00000, 3.00000) -- (2.00000, 3.00000, 3.00000) -- (2.00000, 3.00000, 4.00000) -- (1.00000, 3.00000, 4.00000) -- cycle {};
\fill[facet] (2.00000, 2.00000, 2.00000) -- (2.00000, 3.00000, 3.00000) -- (2.00000, 3.00000, 4.00000) -- (2.00000, 2.00000, 4.00000) -- cycle {};
\fill[facet] (2.00000, 2.00000, 2.00000) -- (1.00000, 1.00000, 1.00000) -- (1.00000, 1.00000, 4.00000) -- (2.00000, 2.00000, 4.00000) -- cycle {};
\draw[edge] (1.00000, 1.00000, 1.00000) -- (1.00000, 1.00000, 4.00000);
\draw[edge] (1.00000, 1.00000, 1.00000) -- (2.00000, 2.00000, 2.00000);
\draw[edge] (2.00000, 3.00000, 4.00000) -- (2.00000, 3.00000, 3.00000);
\draw[edge] (2.00000, 3.00000, 4.00000) -- (2.00000, 2.00000, 4.00000);
\draw[edge] (2.00000, 3.00000, 4.00000) -- (1.00000, 3.00000, 4.00000);
\draw[edge] (2.00000, 3.00000, 3.00000) -- (2.00000, 2.00000, 2.00000);
\draw[edge] (2.00000, 3.00000, 3.00000) -- (1.00000, 3.00000, 3.00000);
\draw[edge] (1.00000, 1.00000, 4.00000) -- (2.00000, 2.00000, 4.00000);
\draw[edge] (1.00000, 1.00000, 4.00000) -- (1.00000, 3.00000, 4.00000);
\draw[edge] (2.00000, 2.00000, 4.00000) -- (2.00000, 2.00000, 2.00000);
\draw[edge] (1.00000, 3.00000, 4.00000) -- (1.00000, 3.00000, 3.00000);
\node[vertex] at (1.00000, 1.00000, 1.00000)     {};
\node[vertex] at (2.00000, 3.00000, 4.00000)     {};
\node[vertex] at (2.00000, 3.00000, 3.00000)     {};
\node[vertex] at (1.00000, 1.00000, 4.00000)     {};
\node[vertex] at (2.00000, 2.00000, 4.00000)     {};
\node[vertex] at (2.00000, 2.00000, 2.00000)     {};
\node[vertex] at (1.00000, 3.00000, 4.00000)     {};
\node[vertex] at (1.00000, 3.00000, 3.00000)     {};
\end{tikzpicture}
\qquad
\begin{tikzpicture}
	[x={(-0.268062cm, -0.535732cm)},
	y={(0.963402cm, -0.149036cm)},
	z={(-0.000034cm, 0.831131cm)},
	scale=1.000000,
	back/.style={loosely dotted, thin},
	edge/.style={color=black, thick},
	facet/.style={fill=andresblue,fill opacity=0.500000},
	vertex/.style={inner sep=1pt,circle,draw=andrespink,fill=andrespink,thick}]
%
%
\coordinate (1.00000, 1.00000, 1.00000) at (1.00000, 1.00000, 1.00000);
\coordinate (2.00000, 4.00000, 6.00000) at (2.00000, 4.00000, 6.00000);
\coordinate (2.00000, 4.00000, 4.00000) at (2.00000, 4.00000, 4.00000);
\coordinate (2.00000, 2.00000, 6.00000) at (2.00000, 2.00000, 6.00000);
\coordinate (2.00000, 2.00000, 2.00000) at (2.00000, 2.00000, 2.00000);
\coordinate (1.00000, 1.00000, 6.00000) at (1.00000, 1.00000, 6.00000);
\coordinate (1.00000, 4.00000, 6.00000) at (1.00000, 4.00000, 6.00000);
\coordinate (1.00000, 4.00000, 4.00000) at (1.00000, 4.00000, 4.00000);
\draw[edge,back] (1.00000, 1.00000, 1.00000) -- (1.00000, 4.00000, 4.00000);
\fill[facet] (2.00000, 2.00000, 2.00000) -- (2.00000, 4.00000, 4.00000) -- (2.00000, 4.00000, 6.00000) -- (2.00000, 2.00000, 6.00000) -- cycle {};
\fill[facet] (1.00000, 4.00000, 4.00000) -- (2.00000, 4.00000, 4.00000) -- (2.00000, 4.00000, 6.00000) -- (1.00000, 4.00000, 6.00000) -- cycle {};
\fill[facet] (1.00000, 4.00000, 6.00000) -- (2.00000, 4.00000, 6.00000) -- (2.00000, 2.00000, 6.00000) -- (1.00000, 1.00000, 6.00000) -- cycle {};
\fill[facet] (2.00000, 2.00000, 6.00000) -- (1.00000, 1.00000, 6.00000) -- (1.00000, 1.00000, 1.00000) -- (2.00000, 2.00000, 2.00000) -- cycle {};
\draw[edge] (1.00000, 1.00000, 1.00000) -- (2.00000, 2.00000, 2.00000);
\draw[edge] (1.00000, 1.00000, 1.00000) -- (1.00000, 1.00000, 6.00000);
\draw[edge] (2.00000, 4.00000, 6.00000) -- (2.00000, 4.00000, 4.00000);
\draw[edge] (2.00000, 4.00000, 6.00000) -- (2.00000, 2.00000, 6.00000);
\draw[edge] (2.00000, 4.00000, 6.00000) -- (1.00000, 4.00000, 6.00000);
\draw[edge] (2.00000, 4.00000, 4.00000) -- (2.00000, 2.00000, 2.00000);
\draw[edge] (2.00000, 4.00000, 4.00000) -- (1.00000, 4.00000, 4.00000);
\draw[edge] (2.00000, 2.00000, 6.00000) -- (2.00000, 2.00000, 2.00000);
\draw[edge] (2.00000, 2.00000, 6.00000) -- (1.00000, 1.00000, 6.00000);
\draw[edge] (1.00000, 1.00000, 6.00000) -- (1.00000, 4.00000, 6.00000);
\draw[edge] (1.00000, 4.00000, 6.00000) -- (1.00000, 4.00000, 4.00000);
\node[vertex] at (1.00000, 1.00000, 1.00000)     {};
\node[vertex] at (2.00000, 4.00000, 6.00000)     {};
\node[vertex] at (2.00000, 4.00000, 4.00000)     {};
\node[vertex] at (2.00000, 2.00000, 6.00000)     {};
\node[vertex] at (2.00000, 2.00000, 2.00000)     {};
\node[vertex] at (1.00000, 1.00000, 6.00000)     {};
\node[vertex] at (1.00000, 4.00000, 6.00000)     {};
\node[vertex] at (1.00000, 4.00000, 4.00000)     {};
\end{tikzpicture}

    \caption{The weakly increasing $\mathbf{x}$-parking function polytopes, from left to right: $\mathfrak{X}_3^w(1,1)$,  $\mathfrak{X}_3^w(1,2)$,  $\mathfrak{X}_3^w(2,1)$,  $\mathfrak{X}_3^w(2,2)$. Note that when $a=1$, they are two-dimensional, and when $a>1$, they are three-dimensional.}
    \label{fig:weakly}
\end{figure}

\begin{proposition}\label{prop:integral_equivalence}
The weakly increasing $\mathbf{x}$-parking function polytope $\X^w_n(a,b)$ is integrally equivalent to the Pitman-Stanley polytope $\mathsf{PS}_n(a-1,b,\dots, b)$.
\end{proposition}

\begin{proof}
Let $T:\R^{n}\to\R^n$ be the linear transformation defined $$T(\bx)=(x_1 - 1, x_2-x_1,x_3-x_2,\dots, x_{n}-x_{n-1}). $$
Note that if $\bx\in \X_n^w(a,b)\cap \Z^{n}$, it follows that $\bx$ is also a weakly increasing $\mathbf{x}$-parking function.
We can see $T(\bx) = \by \in  \mathsf{PS}_n(a-1,b,\dots, b)\cap \Z^n$ since $1\leq x_i \leq x_{i+1}$ implies $y_i \geq 0$ and $1\leq x_i \leq a+(i-1)b$ implies $\sum_{i=1}^k y_i = x_k - 1\leq a+(i-1)b  -1$ for all $i$.

Next, define the linear transformation $S:\R^n\to \R^{n}$ by \[S(\by)=(1+y_1,1+y_1+y_2,\dots, 1+y_1+\dots +y_n).\]
For $\by\in \mathsf{PS}_n\cap \Z^n$, we have that $\bx = S(\by)$ satisfies $x_i\leq x_{i+1}$  and $$x_i = 1 + \sum_{k=1}^i y_k \leq 1+ (a-1) + (i-1)b = a +(i-1)b,$$ 
hence $\bx\in \X^w_n(a,b)\cap \Z^{n}$.
By construction, both $T$ and $S$ are injective.
\end{proof}

\begin{corollary}
Let $t\in \Z_{\geq 0}$. 
The number of lattice points in the $t$-dilate of  $\X^w_n(a,b)$ is given by 
 $$|t\X^w_n(a,b)\cap \Z^n|=\frac{1}{n!}(t(a-1)+1)(t(a-1+nb)+2)(t(a-1+nb)+3)\cdots (t(a-1+nb)+n).$$
\end{corollary}

\begin{proof}
By Proposition \ref{prop:integral_equivalence}, $\X^w_n(a,b)$ is integrally equivalent to $\mathsf{PS}_n(a-1,b,\dots, b)$.
By substituting $a-1$ for $a$ in the equation of Theorem 13 in \cite{PitmanStanley}, we obtain the result. 
\end{proof}

Another consequence of Proposition \ref{prop:integral_equivalence} is the following:

\begin{corollary}
For the special case when $\bx=(a,b,\dots,b)=(1,1,\dots,1)\in \R^n$, the weakly increasing classical parking function $\X_{n}^w(1,1)$ has volume $n^{n-2}$ and contains $C_{n}$ lattice points, where $C_n=\displaystyle{\frac{1}{n+1}\binom{2n}{n}}$ denotes the $n$-th Catalan number. 
\end{corollary}

\begin{proof}
The weakly increasing classical parking function $\X^w_n(1,1)$ is integrally equivalent to the Pitman-Stanley polytope $\mathsf{PS}_{n}(0,1,1,\ldots,1) = \mathsf{PS}_{n-1}(1,1,\ldots,1)$. 
It follows from work by Pitman and Stanley \cite{PitmanStanley} and Benedetti et al.~\cite{BGHHKMY} that the Pitman-Stanley polytope is integrally equivalent to a flow polytope arising from a graph consisting of a path $1\rightarrow 2 \rightarrow \cdots \rightarrow n$ and additional edges $(1,n),(2,n),\dots, (n-2,n)$, for which the volume and lattice point count are known.
\end{proof}

The following proposition known and can be deduced from the existing literature detailed above, but we provide a proof for completeness.

\begin{proposition}
The weakly increasing classical parking function polytope $\X^w_n(1,1)$ is an $(n-1)$-dimensional polytope given by the following equality and inequalities.
\begin{align*}
    x_1  &= 1,\\
    x_i &\leq i, \text{ \emph{for} } 2\leq i\leq n,\\
    x_{i-1} &\leq x_i,  \text{ \emph{for} }2\leq i\leq n.
\end{align*}
Furthermore, $\X_n^w(1,1)$ has
\begin{enumerate}[\emph{(}i\emph{)}]
    \item $2(n-1)$ facets,
    \item $2^{n-1}$ vertices, and 
    \item $2^{n-2}(n-1)$ edges.
\end{enumerate}
\end{proposition}

\begin{proof}
The inequality description follows similar to the proof of Proposition \ref{prop:inequality} restricting ourselves to the weakly-increasing $\bx$-parking functions. 
\begin{enumerate}[($i$)]
    \item This follows from a straightforward enumeration of the inequalities in the description above. 
    \item All the vertices are of the form $$(\underbrace{1,\ldots,1}_k, v_{k+1}, \dots, c_n),$$ where  $v_k = k$ and for each $k \leq i < n$ we have either  $v_{i+1} = v_i$ (``min'') or $v_{i+1} =i+1$ (``max'').  
    Hence, each vertex corresponds to a sequence of $n-1$ binary choices.
    \item We claim that vertices $v,u$ are connected by an edge if and only if their construction (given above by a sequence of binary choices) differs by exactly one choice.

We show this inductively.
Assume this is true for $\X_m^w (1,1)$ where $m<n$ and let $u,v$ be two vertices in $\X_n^w(1,1)$ that differ by the $k$-th choice.
If $k = n$, we have that $u' = v'$ where $v' = (v_1, \dots, v_{n-1})$, which is a vertex of $\X_{n-1}^w (1,1)$. 
Hence, there exists a $c'\in \mathbb{R}^{n-1}$ that maximizes it. 
Let $c = (c',0)\in \mathbb{R}^n$; it is evident that $c\cdot u = c\cdot v > c \cdot w$ for any other vertex $w\in \X_n^w(1,1)$.
If $k<n$ and we continue to choose min for the rest of the choices after $k$, it follows that
\begin{align*}
    u &= (u_1, \dots, u_{k-1}, \underbrace{u_{k-1}, \dots, u_{k-1}}_{n-k+1}), \\ v &= (u_1, \dots, u_{k-1}, \underbrace{k, \dots, k}_{n-k+1}).
\end{align*}
\end{enumerate}
\end{proof}

\begin{remark}
It follows from the inequality description of $\X_n^w(1,1)$ that all lattice points are weakly increasing parking functions.
\end{remark}


\section{Further Directions and Discussion}\label{sec:future}
We conclude this paper by providing some directions for future research. 

\subsection{On the classical parking function polytope}

Given a term order $\prec$, every non-zero polynomial $f \in k[\mathbf{x}]$ has a unique initial monomial, denoted by $in_\prec(f)$. 
If $I$ is an ideal in $k[\mathbf{x}]$, then its \textit{initial ideal} is the monomial ideal $in_\prec(I) := \langle in_\prec(f) : f \in I \rangle$.
Let $\mathcal{A} = \{ \mathbf{a}_1, \ldots, \mathbf{a}_k\} \subseteq \Z^n$ and denote the toric ideal of $\mathcal{A}$ by $I_\mathcal{A}$. 

\begin{proposition}[Corollary 8.9, \cite{Stu}] The initial ideal $in_\prec(I_\mathcal{A})$ is square-free if and only if the corresponding regular triangulation $\Delta_\prec$ of $\mathcal{A}$ is unimodular.
\end{proposition}

Computational evidence suggests the following conjecture, which could be approached using the theory of Gr\"obner bases.

\begin{conjecture}
The parking function polytope $\PF_n$ admits a regular unimodular triangulation. 
\end{conjecture}

\begin{example}
Consider all lattice points of $\PF_3$, which includes the $16$ parking functions of length $3$ and the point $(2,2,2)$ which can be used for a triangulation.
Let $R = \Q[a,b, \ldots,p,q]$, where each variable corresponds to a lattice point of $\PF_3$, and $S = \Q[x,y,z,w]$, and take $f: S \to R$ to be defined by
\[(x,y,z,w) \mapsto (x^1 y^1 z^1 w, x^1 y^1 z^2 w, x^1 y^2 z^1 w, x^2 y^1 z^1 w, x^1 y^1 z^3 w, x^1 y^3 z^1 w, x^3 y^1 z^1 w, x^1 y^2 z^2 w, \]
\[x^2 y^1 z^2 w, x^2 y^2 z^1 w,x^1 y^2 z^3 w, x^1 y^3 z^2 w, x^2 y^1 z^3 w, x^3 y^1 z^2 w, x^3 y^2 z^1 w, x^2 y^3 z^1 w,  x^2 y^2 z^2 w). \]

Then the initial ideal is 
\[(ae, af, bf, ef, ag, bg, cg, eg, fg, ah, bh, gh, ai, bi, ci, fi, aj, bj, cj, ej, ak, bk, ck, dk, fk, gk, al, bl,  cl, \]
\[ dl, el, gl, il, am, bm, cm, dm, fm, gm, hm, an, bn, cn, dn, en, fn, hn, kn, ln, ao, bo, co, do, eo, fo,  \]
\[ho, io, ko, lo, mo, ap, bp, cp, dp, ep, gp, hp, ip, kp, mp, np, aq, bq, cq, dq, eq, fq, gq, hq, iq, kq).\]
Notice that the initial ideal is square-free; hence, there exists a unimodular triangulation of this parking function polytope using on the parking functions and an additional lattice point as vertices.
\end{example} 

If this conjecture holds true, the following problem may be of interest.

\begin{problem}
Find a bijection between the simplices of a unimodular triangulation of $\PF_n$ and $(0,1)$-matrices with two 1's in each row with positive permanent, as discussed in Theorem \ref{thm:main_theorem}.
\end{problem}

The \emph{Ehrhart function} of a polytope $P\subset \R^n$ is $\operatorname{ehr}_P(t):=|tP\cap \Z^n|$, where $tP=\{t\bx:\ \bx\in P\}$.
When $P$ is a lattice polytope (its vertices have integer coordinates), the Ehrhart function is a polynomial in $t$, with degree equal to the dimension of $P$, leading coefficient equal to its normalized volume, second-leading coefficient equal to half the surface area, and constant coefficient 1.
The \emph{Ehrhart polynomial} of a lattice polytope $P$ of dimension $n$ can always be written in the form $\operatorname{ehr}_P(t)=\sum_{i=0}^nh_i^*\binom{t+n-i}{n}$; the sequence $(h_0^*,\dots,h_n^*)$ is called the \emph{$h^*$-vector}. 
Equivalently, $\sum_{t\geq 0}\operatorname{ehr}_P(t)z^t=\frac{h^*(P;z)}{(1-z)^{n+1}}$, where $h^*(P;z)=h_0^*+h_1^*z+\cdots+h_n^*z^n$.

\begin{problem}
Determine a formula for the Ehrhart polynomial (or equivalently, the $h^*$-polynomial) of $\PF_n$.   
\end{problem}

If the conjecture above holds true, then it may be useful in studying the $h^*$-polynomial due to the following proposition due to Stanley \cite{StanleyDecompositions}, which would require us instead to study the $h$-polynomial of the triangulation. 

\begin{proposition}
If $P$ is a lattice polytope that admits a unimodular triangulation, then the $h^*$-polynomial is given by the $h$-polynomial of the triangulation.     
\end{proposition}

\subsection{On \texorpdfstring{$\mathbf{x}$}{\textbf{x}}-parking function polytopes}\label{sec:general-x}
The main object of study in this paper is the $\bx$-parking function polytope for $\bx=(a,b,\dots,b)$. 
One can ask for the face structure, volume, and lattice-point enumeration for parking function polytopes where $\bx\neq (a,b,\dots,b)$.
In the same spirit as Stanley's original problem \cite{Sta}, we pose the following:

\begin{problem}\label{problem_x} For $\bx=(x_1,\dots,x_n)\in \Z_{>0}^n$, let $\mathfrak{X}_n$ be the convex hull of $\bx$-parking functions of length $n$. 
\begin{enumerate}
    \item Find the number of $k$-dimensional faces of $\mathfrak{X}_n$ for $k \in \{ 0, \ldots, n\}$ and given $\bx$. 
    \item Find the volume of $\mathfrak{X}_n$ for given $\bx$. 
    \item Find the number of integer points in $\mathfrak{X}_n$ for given $\bx$, i.e., the number of elements of $\mathfrak{X}_n \cap \Z^n$.
    \item More generally, find a formula for the Ehrhart polynomial (or equivalently, the $h^*$-polynomial) of $\mathfrak{X}_n$ for given $\bx$.
\end{enumerate}    
\end{problem}
This problem is likely to be challenging in its full generality, as even finding explicit formulas for the number of $\mathbf{x}$-parking functions for arbitrary $\mathbf{x}$ is noted by Yan to be challenging \cite{Yan2}. 
General enumerative results can  be found in \cite{KY} in terms of Gon\v{c}arov polynomials and more on the history of enumerative results can be found in \cite{GH}.
A starting point for Problem \ref{problem_x} could be to consider some of those special cases from \cite{Yan2} and explore the connection to Gon\v{c}arov polynomials.

We note that the problem to determine the number of lattice points and the Ehrhart polynomial (or $h^*$-polynomial) for the $\bx$-parking function polytope when $\bx=(a,b,\dots,b)$, remains open. 

\subsection{Other generalizations of parking functions and their convex hulls}
There are a plethora of generalizations of parking functions in the literature and one can naturally ask similar problems to Problem \ref{problem_x} for their favorite generalization.  
One such generalization is known as the $(a,b)$-\emph{rational parking functions}. 

There is a well-known bijection between Dyck paths of length $n$ and all possible increasing rearrangements of parking functions of length $n$. 
By labeling the North steps of the Dyck paths with elements in $[n] = \{1,2,\dots, n\}$ such that each element appears exactly once and the labels increase within each column going North, we can construct a bijection between the labeled Dyck paths of length $n$ and all parking functions of length $n$.
Now, let $a,b \in \Z_{>0}$.
An $(a,b)$-\emph{Dyck path} is a lattice path from $(0,0)$ to $(b,a)$ (that is, with $a$ North steps and $b$ East steps) which stays weakly above the diagonal line $y = \frac{a}{b}x$. 
A $(n,n)$-Dyck path is just a standard Dyck path of length $n$. 
There is a canonical bijection between $(n,n)$-Dyck paths and $(n,n+1)$-Dyck paths since the last step of a $(n,n+1)$-Dyck paths must be an East step.
If $a,b$ are coprime, the number of $(a,b)$-Dyck paths is given by $$\frac{1}{a+b}\binom{a+b}{a,b} = \frac{(a+b-1)!}{a!b!},$$ and is called the \emph{rational Catalan number}.

\begin{definition}
Let $a,b$ be coprime.
An $(a,b)$-\emph{parking function} is an $(a,b)$-Dyck path together with a labeling of the North steps by the set $[a] = \{1,2,\dots, a\}$ such that the labels increase within each column going North. 
Define the $(a,b)$-\emph{parking function polytope} $\mathcal{P}_{a,b}$ as the convex hull of all $(a,b)$-parking functions of length $n$ in $\R^n$.
\end{definition}

\begin{remark}
There are $b^{a-1}$ many $(a,b)$-parking functions for coprime $(a,b)$.
The classical parking functions are recovered when $(a,b) = (n,n+1)$.
\end{remark}

The following are two propositions towards the study of the $(a,b)$-\emph{parking function polytope}. 

\begin{proposition}\label{prop:rational_vertices}
Consider the $(a,b)$-parking function polytope $\mathcal{P}_{a,b}$.
\begin{enumerate} 
    \item $\mathcal{P}_{a,b}$ is an $a$-dimensional polytope where the vertices are permutations of 
\[(\underbrace{1,\ldots, 1\,}_k, b_{k+1}, b_{k+2}, \ldots, b_a),\]
for $1 \leq k \leq a$ where $b_i = \lceil \frac{b}{a}(i-1)\rceil$ for $i>1$, $b_1 = 1$.

\item If $b>a$, then the number of vertices of  $\mathcal{P}_{a,b}$  is 
    \[a! \left(\frac{1}{1!} + \frac{1}{2!} + \cdots + \frac{1}{a!}\right).\]
 If $b<a$, for any $1\leq i\leq b$, let $m_i = | \{j \text{ such that } b_j = i\}|$.  Then the number of vertices of $\mathcal{P}_{a,b}$ is 
        \[a!  \left(\frac{1}{m_1!m_2!\cdots m_b!} + \sum\limits_{k=2}^b M_k\right),\]
where for $2\leq k \leq b$, $$M_k = \frac{1}{m_{k+1}!\cdots m_b!} \left(\sum\limits_{i=1}^{m_k} \frac{1}{(m_1+\cdots + m_{k-1} + i)! (m_k-i)!}\right).$$
\end{enumerate}   
\end{proposition}

\begin{proof} [Proof sketch]
   Any $(a,b)$-rational parking function is a permutation of the weakly increasing sequence $(\alpha_1, \alpha_2,\dots, \alpha_a)$ where $\alpha_i$ denotes the column in which the $i$-th north step happens (where we count the column $(x,0)$ as the 1st column). 
   Since this path must be above the line $y=\frac{a}{b}x$ and the north steps are from $(\alpha_i -1, i-1)$ to $ (\alpha_i -1, i)$, it follows that $\alpha_i  -1 \leq  \frac{b}{a}(i-1)$. 
   This implies $\alpha_i \leq \lceil \frac{b}{a}(i-1) \rceil$.
Thus all $(a,b)$-rational parking functions are permutations of the weakly increasing sequence   $(\alpha_1, \alpha_2,\dots, \alpha_a)$ where $\alpha_i \leq \lceil \frac{b}{a}(i-1) \rceil$.
Using a similar argument to that of Proposition \ref{prop:abn_vertex}, this implies that all vertices of $\mathcal{P}_{a,b}$ are permutations of 
$$(\underbrace{1,\cdots, 1\,}_k, b_{k+1}, b_{k+2}, \dots, b_a)$$
where $b_i = \text{max}(\alpha_i)$ for $i>1$. 
This is equivalent to saying $b_i =  \lceil \frac{b}{a}(i-1) \rceil$. 
\end{proof}

\begin{proposition}\label{prop:rational_hyperplane}
The $(a,b)$-parking function polytope $\mathcal{P}_{a,b}$ can be described by the following  inequalities:

\begin{align*}
    1&\leq x_i\leq b_a, &\text{ for } 1\leq i \leq a,\\
    x_i+x_j &\leq b_{a-1} + b_a, &\text{ for } i<j,\\
    &\vdots \\ 
    x_{i_1} + x_{i_2} + \cdots + x_{i_{a-2}} &\leq b_3 + b_4 +\cdots + b_{a-1}+  b_a, &\text{ for }  i_1 < i_2 < \cdots < i_{a-2}, \\ 
    x_1+x_2+\cdots + x_a &\leq b_1 + b_2 + \cdots + b_{a}.
\end{align*}
Note that in the case where $b>a$, this inequality description is minimal and the number of facets is equal to $2^a -1$. 
  
\end{proposition}
\begin{proof}[Proof Sketch]
    We can see that the first collection of inequalities describes any $(a,b)$-parking function, using the vertex description above and the argument given in the proof of Proposition \ref{prop:inequality}. 
    In the case where $b>a$, the same argument shows that the description is minimal, since all the $b_i$ are distinct. 
\end{proof}

\begin{remark}
    In the case where $b<a$, this collection of hyperplanes may contain redundancies, since not all $b_i$ are distinct. For example, in the case where $b = a-1$, the minimal hyperplane description is
\begin{align*}
    1&\leq x_i\leq b_a, &\text{ for } 1\leq i \leq a,\\ 
    x_i+x_j &\leq b_{a-1} + b_a, &\text{ for } i<j,\\ 
    &\vdots \\ 
    x_{i_1} + x_{i_2} + \cdots + x_{i_{a-3}} &\leq  b_4 + b_5+\cdots + b_{a-1}+  b_a, &\text{ for }  i_1 < i_2 < \cdots < i_{a-3}, \\ 
    x_1+x_2+\cdots + x_a &\leq b_1 + b_2 + \cdots + b_{a}.
\end{align*}
From this, we can see that the number of facets is equal to  $2^a -1 - \binom{a}{a-2} = 2^a -\frac{(a-2)(a+1)}{2}$.
\end{remark}


\section*{Acknowledgements}

The authors thank Esme Bajo and Jason Zhao for helpful conversations and Douglas Varela for his insight on the analytical tools used. We also thank the referees for helpful feedback.

\section*{Conflict of Interest Statement}
On behalf of all authors, the corresponding author states that there is no conflict of interest.

\bibliographystyle{amsplain}
\bibliography{biblio}

\end{document}